\documentclass[11pt]{article}
\usepackage[utf8]{inputenc}
\usepackage[english]{babel}
\usepackage{amsmath}
\usepackage{amsthm}
\usepackage[usenames]{color}
\usepackage{amsfonts}
\usepackage{amssymb}
\usepackage{makeidx}
\usepackage{graphicx}
\usepackage[left=2.6cm,right=2.6cm,top=3.5cm,bottom=2cm]{geometry}
\usepackage{enumerate}
\usepackage{mathrsfs}
\usepackage{multicol}

\usepackage{todonotes}

\newcommand{\growth}{\mathrm{gr}}
\newcommand{\grH}{D_{H;G}}

\usepackage[backref=page]{hyperref}

\newtheorem{theorem}{Theorem}[section]

\newtheorem{thmx}{Theorem}

\newtheorem{prop}[theorem]{Proposition}

\newtheorem{lemma}[theorem]{Lemma}

\newtheorem{corollary}[theorem]{Corollary}

\newtheorem{definition}[theorem]{Definition}

\newtheorem{remark}[theorem]{Remark}

\newtheorem{example}[theorem]{Example}

\providecommand{\abs}[1]{\left|#1\right|}

\renewcommand{\Im}{\mathop{\mathrm{Im}}}

\newcommand{\N}{\mathbb{N}}
\newcommand{\Z}{\mathbb{Z}}
\newcommand{\R}{\mathbb{R}}
\newcommand{\p}{\mathbb{P}}
\newcommand{\E}{\mathbb{E}}
\newcommand{\Prob}{\p}

\newcommand{\Class}{\mathcal{S}}
\newcommand{\StSub}{\mathscr{S}}
\newcommand{\dist}{\mathrm{dist}}
\newcommand{\SL}{\mathrm{SL}}

\newcommand{\mF}{\mathcal{F}}
\newcommand{\Orb}{\mathcal{O}}

\newcommand{\ta}{\widetilde{a}}
\newcommand{\tH}{\widetilde{H}}
\newcommand{\tq}{\widetilde{q}}

\numberwithin{equation}{section}

\newcommand{\GrFn}{\mathcal{V}}
\newcommand{\wT}{\widetilde{T}}
\newcommand{\wL}{\widetilde{L}}
\newcommand{\wtp}{\widetilde{p}}

\newcommand{\eps}{\varepsilon}
\newcommand{\lh}{l}

\newcommand{\wt}{\mathrm{wt}}

\newcommand{\Sc}{\mathbb{S}^1}

\newcommand{\const}{\mathrm{const}}

\newcommand{\Crit}{\mathop{\mathrm{Crit}}\nolimits}
\newcommand{\Diff}{\mathrm{Diff}}
\newcommand{\id}{\mathrm{id}}
\newcommand{\Ker}{\mathop{\mathrm{Ker}}\nolimits}
\newcommand{\Tor}{\mathop{\mathrm{Tor}}\nolimits}
\newcommand{\Hom}{\mathrm{Hom}}
\newcommand{\Homeo}{\mathrm{Homeo}}

\newcommand{\Fix}{\mathop{\mathrm{Fix}}\nolimits}

\newcommand{\Stab}{\mathrm{Stab}}
\newcommand{\osc}{\mathop{\mathrm{osc}}}

\newcommand{\rk}{\mathop{\mathrm{rk}}}

\newcommand{\reg}{\tau}

\newcommand{\bd}{\mathcal{D}}
\newcommand{\xvl}{x_v^-}
\newcommand{\xvr}{x_v^+}

\title{Critical regularity of nilpotent groups acting on\\  one-dimensional compact manifolds} 

\begin{document}

\author{\textsc{Maximiliano Escayola and Victor Kleptsyn }}

\date{}

\maketitle

\begin{abstract} 

Given a finitely generated, torsion-free nilpotent group, we find the maximum possible (\emph{critical}) regularity for its faithful actions by diffeomorphisms of the closed or half-open interval and of the circle. Our result gives an expression for its value in purely algebraic terms (using the relative growth of appropriate subgroups), generalizing many preceding works. As an intermediate step, we generalize the Bass-Guivarc'h formula, obtaining a formula for the relative growth of subgroups of nilpotent groups, as well as for the growth of the corresponding Schreier graphs.

\end{abstract}
\vspace{0.2cm}


\vspace{0.2cm}


\vspace{0.7cm}



\tableofcontents


\section{Introduction}

\subsection{History of the question}

The present paper is devoted to the study of critical regularity (that is, the supremum of all regularities of admissible actions) for the actions of nilpotent groups on one-dimensional manifolds. The critical regularity of actions have been studied for a long time, with some variations in the context. For instance, one can consider all faithful actions, or restrict oneself to the actions conjugate to a given one, or to a given class of actions. 

The Denjoy theorem~\cite{Denjoy} states that an action of $\Z$ of the circle that is of class $C^2$ cannot admit a Cantor minimal set, while examples, constructed by M.~Herman~\cite{Herman}, show that this is possible in any regularity $C^{1+\tau}$ for any $\tau<1$. 
In~\cite{DKN_acta}, it was shown that the critical regularity for the generalization of the Denjoy theorem to the actions of $\Z^d$ on the circle is equal to $1+1/d$; that is, for every $\tau<1/d$ there are faithful actions of class $C^{1+\tau}$ with a Cantor minimal set, while for any $\tau>1/d$ it is impossible.

For general non abelian nilpotent groups acting on the interval~$[0,1]$, impossibility of a $C^2$-action (and thus an upper bound for the critical regularity) was shown by J.~Plante and W.~Thurston in their work~\cite{pt} (see Theorem~4.5 therein), and a lower bound of the form $1+1/d$, where $d$ is the exponent of the polynomial growth of the group, was established by K.~Parkhe~\cite{parkhe}. This lower bound is not optimal: the critical regularity for $N_4$ (that is, the group of upper-triangular matrices with integer elements and $1$'s on the diagonal) was shown by E.~Jorquera, A.~Navas, C.~Rivas~\cite{JNR} to be equal to~$1+1/2$, higher than the lower bound mentioned above. The question of determining possible regularity for nilpotent group actions appears in~\cite[Section 4]{na-quest} (see also \cite[pp.~990-991]{navas-intermediate}).

Regularity of actions of metabelian groups was studied in~\cite{ER}, where the actions of any regularity less than $1+1/k$, where $k=\rk G/A$, and $A<G$ a maximal abelian subgroup, were constructed. This regularity was shown to be critical, in particular, for the generalised $(2n+1)$-dimensional Heisenberg groups $H_{2n+1}$. Also, the regularity of actions of nilpotent groups~$N_k$ was studied in works of B.~Farb and J.~Franks~\cite{ff} and of G.~Castro, E.~Jorquera, A.~Navas~\cite{int4}, and it was shown by A.~Navas~\cite{na-crit} that the critical regularity of Farb--Franks actions cannot be attained. Groups of intermediate growth were shown not to have any $C^{1+\alpha}$-actions on $[0,1]$ for any $\alpha>0$ by A.~Navas~\cite{navas-intermediate}, together with an example of such a group that can be realised by $C^1$-diffeomorphisms.

The construction of examples of group actions of a given regularity~$C^{1+\tau}$ was first done using a technique introduced by M.~Herman in his thesis~\cite{Herman}. It was then succeeded by a technique, going back to the work of T.~Tsuboi~\cite{tsuboi}, that is nowadays a standard tool of constructing one-dimensional actions (and that we will be using in the present paper). 


In their famous paper~\cite{kk}, S. H. Kim and T. Koberda have constructed examples of groups of arbitrary critical regularity $\alpha=k+\tau$ for actions on the interval $[0,1]$ and on the circle. Later, K.~Mann and M.~Wolff in~\cite{mann-wolff} have used different techniques for the actions on the interval $[0,1]$, circle $\R/\Z$ and line $\R$. In particular, such regularities (being arbitrary real numbers) might not admit any algebraic expression, let alone being rational numbers.

There have been many other recent papers devoted to the study of critical regularity in different contexts: \cite{bkk,kk-artin}, where the right-angled Artin groups were studied; \cite{rivas-triestino}, constructing actions of Higman's group by homeomorphisms of the line and showing that there are no nontrivial actions by $C^1$-diffeomorphisms; \cite{kkr}, studying the actions of the lamplighter group; \cite{KMBST}, showing the $C^1$-smoothability of one-dimensional actions of groups of subexponential growth, and recent monograph~\cite{BMRT}, where, in particular, non-smoothability of actions of some piecewise-linear groups have been established (see Section~7.3 therein); and many others. 
We refer the reader to the book~\cite{kk_book} for a more detailed survey. 

\subsection{Description of results}

Our paper is devoted to the study of critical regularity for actions of finitely generated nilpotent groups on one-dimensional manifolds: closed interval $[0,1]$, half-interval $(0,1]$ (or, what is the same, half-line $\R_{\le 0}$), circle $\Sc$. Our first main result, Theorem~\ref{thm crit interval} below, provides an algebraic description for the critical regularity for any such group on the closed interval~$[0,1]$. Namely, we show, that the critical regularity of such actions is equal to~$1+1/d$, where $d$ is an integer number (described in terms of a minimax of relative growth of so-called \emph{stabilizer subgroups}, see Section~\ref{s:p-stab} for the details). 

We then consider the case of the half-open interval and the circle. It could have happened that in one of these cases, the critical regularity would be higher than the one for the closed interval~$[0,1]$. Indeed, for the case of the half-interval (or closed half-line), the local H\"older constant for the derivative of a diffeomorphism might grow unboundedly as the point approaches the removed endpoint (respectively, tends to infinity). For the case of the circle, there are more actions than on the interval. However, it turns out (see Theorems~\ref{t:half-open} and~\ref{t:circle} below) that in both of these cases the critical regularity stays the same as for the case of the closed interval.

On the contrary, passing to the open interval, when both endpoints are removed (or, what is the same, considering the actions on~$\R$), might lead to the higher value of the critical regularity. For instance, for the case of the group $N_4$ its critical regularity for the action on~$[0,1]$ was shown in~\cite{JNR} to be equal to~$1+1/2$. Meanwhile, as we show in Example~\ref{ex-N4} below, this group admits $C^{1+\alpha}$-actions on~$\R$ for any~$\alpha<1$. Though, the general description of critical regularity for the actions on~$\R$ becomes more complicated, so we are going to address it in another paper.

%

We also study the topologically free actions on the interval (that is, the sets of fixed points of maps have empty interior). It turns out that the critical regularity restricted to such actions also might be different from the critical regularity for general faithful actions. We show it for the case of the product of two Heisenberg groups $G=N_3\times N_3$ in Example~\ref{ex:N3-twice}. Though again, to limit the technical part of this paper, we are going to address the critical regularity for topologically free actions in another paper.


\subsection{Plan of the paper}

In Section~\ref{s:prelim}, we recall and introduce necessary preliminaries. An experienced reader can safely jump over almost all of it, looking only at Section~\ref{s:p-stab} for the notion of the stabilizer subgroup and details related to it.

We state our main results for the actions in Section~\ref{s:main}. We start with Theorem~\ref{t:A}, addressing a weaker class of actions where only a given central element is required to act nontrivially. However, as is suffices to check the faithfulness of an action of a nilpotent group on central elements only (see Lemma~\ref{l:on-centre}), this already allows us to handle the case of nilpotent groups with \emph{cyclic} center~$Z(G)$, see Corollary~\ref{rk:cyclic}. We then pass to stating our results for the actions of general nilpotent groups on the closed interval~$[0,1]$ (Theorem~\ref{thm crit interval}), on half-open interval (Theorem~\ref{t:half-open}) and on the circle (Theorem~\ref{t:circle}). 

Finally, in Section~\ref{s:gen-Bass} we state our last main result, a generalisation of Bass-Guivarc'h formula, Theorem~\ref{t.B-G}, that provides a formula for the growth of the quotient $G/H$ used in the statements of previous main results. It is proved in Section~\ref{s:proof-E}.

We then proceed to the proof of the regularity estimates, starting with the estimates from above, that are addressed in Section~\ref{s:upper}. 
The key idea here is the construction (in the general case of a nilpotent group) of random process in the group that allows to ensure control of the distortion. For the case of a compact~$M$, this is done in Section~\ref{s:upper-compact}: the process is constructed in Proposition~\ref{p:process}, and it is applied in Proposition~\ref{l:alpha-conv} to obtain the distortion control, then used to prove the upper bound for the regularity via the Generalized Kopell Lemma (Proposition~\ref{p:Kopell}). Then, in Section~\ref{s:half-int} we modify this technique for the case of a half-open interval. Namely,  the modified process (allowing to ensure that the image intervals stay in the compact domain) is constructed in Proposition~\ref{p:process-new}, and Lemma~\ref{l:alpha-conv-new} applies it to obtain a bound for the sum of the $\alpha$-powers of the intermediate lengths (providing the control for the distortion). Finally, we modify the arguments to show that the critical regularity cannot be attained in Section~\ref{s:critical}: a uniform control for the distortion in then replaced by a slowly degrading one (see Proposition~\ref{p:process-critical}), but this already suffices to obtain a contradiction (see Lemma~\ref{l:no-sublinear}).

In Section~\ref{s:lower}, we establish a lower bound for the critical regularity by constructing actions of a class $C^{1+(\frac{1}{d}-\eps)}$. We start by recalling the Pixton-Tsuboi strategy in Section~\ref{s:P-T} and the Tsuboi family of maps, used for its realisation, in Section~\ref{s:T-family}. The next step of the implementation of the strategy is the choice of lengths of the intervals; we provide a general framework for this choice (having an advantage that the resulting maps are, roughly speaking, as regular as they can be) in Section~\ref{s:Choice-and-control}. We then state and prove the realisation statement for a prescribed stabiliser subgroup, Proposition~\ref{p:lower-K}, in Section~\ref{s:concluding-realisation}. We conclude this section by proving in Section~\ref{s:technical} the technical estimates for the H\"older constants, that were postponed from Section~\ref{s:Choice-and-control}.

In Section~\ref{s:main-proofs}, we complete the proofs of the main theorems. The upper bounds are deduced from Proposition~\ref{thm upper bound}, while the lower bounds are deduced from Proposition~\ref{p:lower-K} (with an optimisation over the stabiliser subgroup, joined with gluing actions to ensure the faithfulness).

Finally, in Section~\ref{s:examples} we provide two examples of classes where the critical regularity differs from the one in Theorem~\ref{t:A}. Namely, in Section~\ref{s:open-interval}, we see that the regularity for an action on the open interval (or on the real line) might be higher than for the actions on~$[0,1]$. Then, in Section~\ref{s:top-free} we consider topologically free actions on $[0,1]$, and show that for these actions the critical regularity might be lower than for general faithful actions.

\subsection{Acknowledgements}

The authors are very grateful to Michele Triestino, Nicol\'as Matte Bon, Andr\'es Navas, Sang-hyun Kim, Crist\'obal Rivas for the fruitful discussions and the interest to the paper. We are specially thankful to Andr\'es Navas for proposing the tools to study the attainability of the critical regularity and motivating us to work on it, and to Michele Triestino for his many remarks.

V.K. was supported in part by ANR Gromeov project (ANR-19-CE40-0007) and by Centre Henri Lebesgue (ANR-11-LABX-0020-01).  M.E. was supported in part by ANR Gromeov project (ANR-19-CE40-0007) and by KIAS individual grant no.~MG107601 of the Korea Institute for Advanced Study.

\section{Preliminaries}\label{s:prelim}

\subsection{Critical regularity}

\begin{definition}
Let $\Class\subset \Hom(G,\Homeo_+(M))$ be a (non-empty) class of orientation-preserving actions of a group $G$ on some manifold~$M$. 
The \emph{critical regularity} of the class $\Class$ is defined as 
\[
\Crit(\Class):=\sup \{\alpha \mid \exists \phi\in \Class : \, \phi:G\to \Diff^{\alpha}(M)\}.
\]
Here for $\alpha=k+\tau$, where $k\in \N\cup \{0\}$, $\tau\in [0,1)$, the set $\Diff^{\alpha}(M)$ is defined as the set of $C^k$-diffeomorphisms of $M$ whose $k$-th order derivatives are $\tau$-H\"older. If $M$ is not compact, the H\"older property is required only locally; that is, the increments of $D^k f$ between points of a compact set satisfy a $\tau$-H\"older upper bound, but the H\"older constant in this bound may depend on the set.
\end{definition}

The most classical notion of the critical regularity is the one of faithful actions: 
\begin{definition}
For any manifold $M$ and a group $G$, denote:
\[
\Crit_M(G) := \Crit (\left\{\phi \in \Hom(G,\Homeo_+(M)) \mid \Ker \phi = \{e\} \right\}).
\]
\end{definition}
However, one can also consider other classes of actions: for instance, those conjugate or semi-conjugate to some prescribed action, topologically free actions, etc. We will use this definition for several classes of actions throughout this paper.

\subsection{Generalized Kopell's Lemma and the control of distortion}

Our main tool for establishing the upper bounds for the regularity of an action is 
a generalization of the classical Kopell's lemma~\cite{kopell}, that has appeared in the work of 
Deroin, Kleptsyn and Navas~\cite{DKN_acta}:

\begin{prop}[Generalized Kopell's lemma] \label{p:Kopell}
Let $f_1,\ldots,f_k$ be $C^1$-diffeomorphisms of compact one-dimensional $M$ that commute with a $C^1$-diffeomorphism $g$. Assume that $I_g$ is 
an interval, fixed by $g$, such that $g|_{I_g}\neq \id$. Assume moreover that
for a certain $0 < \reg < 1$ and a sequence of indices $i_j\in \{1,\ldots,k\}$, we have a finite sum
\begin{equation}\label{eq:L-sum}
L_\reg=\sum_{j\geq 0}\abs{f_{i_j}\cdots f_{i_1}(I_g)}^\reg <\infty. 
\end{equation}
Then $f_1,\ldots, f_k$ cannot be all of class $C^{1+\reg}$.  
\end{prop}

For the reader's convenience, we sketch here the ideas of the proof. First, one uses the standard technique to estimate the distortion of the compositions $F_j := f_{i_j}\cdots f_{i_1}$. Namely: 

\begin{definition}
Let $f\in \Diff^1(M)$, and $I\subset M$ be an interval. The \emph{distortion} of $f$ on $I$ is defined as 
\begin{equation}\label{eq:def-dist}
\varkappa(f;I):=\osc_{I} \log Df = \max_{x,y\in I} \log \frac{Df(x)}{Df(y)}.
\end{equation}
\end{definition}
It is a standard remark that the distortion is composition-subadditive,
\begin{equation}\label{eq:kappa-subadditive}
\varkappa(f\circ g;I) \le \varkappa(f;g(I)) +\varkappa(g;I),
\end{equation}
and for a given $C^{1+\reg}$-diffeomorphism $f$ its distortion on an interval of a given length is bounded from above by
\begin{equation}\label{eq:dist-reg}
\varkappa(f;I) \le C_f \cdot |I|^{\reg}.
\end{equation}
Joined with the subadditivity, the bounds~\eqref{eq:dist-reg} and the assumption~\eqref{eq:L-sum} imply that the distortions $\varkappa(F_n;I_g)$ are uniformly bounded. 

\begin{figure}[h!]
\begin{center}
\includegraphics{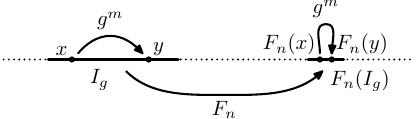}
\end{center}
\caption{Estimating the derivative of~$g^m$.}\label{fig:Fn-gm}
\end{figure}

On the other hand, the derivative of any power $g^m$ at any point $x\in I_g$ can be written as
\begin{equation}\label{eq:Dgm}
Dg^m (x) = D(F_n^{-1} g^m F_n)(x) = Dg^m |_{F_n(x)} \cdot \frac{DF_n (x)}{DF_n(y)},
\end{equation}
where $y=g^m(x)$ and $n$ is arbitrary (see Fig.~\ref{fig:Fn-gm}). The second factor in the right hand side of~\eqref{eq:Dgm} is uniformly bounded due to the distortion control, that is, due to the boundedness of~$\varkappa(F_n;I_g)$. At the same time, for any fixed~$m$ the first factor converges to~$1$ as $n\to \infty$ (as the diffeomorphism $g^m$ preserves the interval $F_n(I_g)$ whose length that tends to~0). Thus, the derivatives of $g^m$ are bounded by a uniform constant. However, for any non-identity diffeomorphism $g$ the derivatives of its iterations cannot be uniformly bounded. This contradiction proves the proposition; we refer the reader to~\cite[Section~3.2]{DKN_acta}, as well as to~ \cite[Proposition 2.1]{int4}, for the complete proof.

Finally, let us introduce a related notion that generalises the estimate~\eqref{eq:dist-reg} above.
\begin{definition}
For a $C^1$-diffeomorphism $f:I\to J$, let $\kappa_\alpha(f,I)$ be the $\alpha$-H\"older constant of $\log Df$ on $I$, provided that it is finite, and $+\infty$ otherwise:
\begin{equation}
\kappa_{\alpha}(f,I):= \sup_{x,y\in I, \atop x\neq y} \frac{|\log Df(x) - \log Df(y)|}{|x-y|^{\alpha}}.
\end{equation}
\end{definition}
Then, $\kappa_{\alpha}(f,I)<\infty$ if and only if $f$ is a $C^{1+\alpha}$-diffeomorphism, and in that case, one has 
\begin{equation}
\forall I'\subset I \quad \varkappa(f,I')\le \kappa_{\alpha}(f;I) \cdot |I'|^{\alpha}.
\end{equation}
This notation will be used later, in Section~\ref{s:lower}, as well as a standard remark
\begin{equation}\label{eq:kappa-composition}
\kappa_\alpha(g\circ f,I)\le \kappa_\alpha(f,I) + \kappa_\alpha(g, f(I)) \cdot \|Df\|_{C(I)}^{\alpha}.
\end{equation}

\subsection{Stabilizer subgroups and Schreier graphs}\label{s:p-stab}

In order to apply the generalized Kopell's Lemma, one needs to have a map $g$, commuting with different elements $f_1,\dots, f_k\in G$, and to have intervals that are fixed by it. A natural choice then is to take an element of the center of the group, $g\in Z(G)$, and to consider one of its support intervals: 
\begin{definition}
Let $G$ be a group acting by homeomorphisms of a connected one-manifold $M$. An interval $I\subseteq M$ is a \emph{support interval} for an element $f\in G$, if it is the closure of a connected component of the complement $M\setminus \Fix(f)$. For the case $M=\Sc$, this definition only applies if $f$ has at least one fixed point; in the case of single-point $\Fix(f)=\{p\}$, this definition gives the full circle (as the closure of an open arc $\Sc\setminus \{p\}$).
\end{definition}

Indeed, choosing $g\in Z(G)$ guarantees that it commutes with all of its elements. Meanwhile, it is easy to see the following:

\begin{lemma}
Let $g\in Z(G)$, and let $I_g$ be a support interval for~$g$. Then for any $f\in G$ the image~$f(I_g)$ is also a support interval of~$g$ (see Fig.~\ref{fig:support}). Two different support intervals of the same~$g$ have disjoint interiors.
\end{lemma}

\begin{figure}[h!]
\begin{center}
\includegraphics{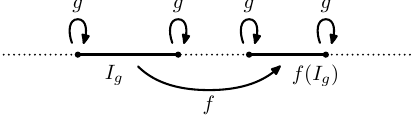}
\end{center}
\caption{A support interval and its image}\label{fig:support}
\end{figure}

The group $G$ is thus acting on the set of support intervals of $g\in Z(G)$. As always, the orbit of an element~$I_g$ of this set is naturally identified with the (left) coset space $G/K$ for the stabilizer subgroup $K=\Stab(I_g)$, where the interval $f(I_g)$ corresponds to the coset~$fK$.

Fixing a set of generators $f_1,\dots,f_k$ of $G$ allows to transform the orbit of $I_g$ into a \emph{Schreier graph}: its vertices are the intervals that form the orbit, and an interval $J$ is linked by edges to $f_1(J),\dots,f_k(J)$. 

Note that when $M$ is compact (that is, the closed interval or the circle), the disjointness of intervals of the orbit implies that the sum of their lengths is finite (it doesn't exceed the total length of~$M$). Hence, a task to apply the generalized Kopell's lemma (to establish an upper bound for the regularity) becomes the task of finding a path in the Schreier graph, along which the sum of powers of lengths~\eqref{eq:L-sum} converges. The obstruction we obtain on this way depends on the structure of the Schreier graph $G/K$, and hence on the subgroup~$K$.

Hence it is natural to ask what subgroups $K\leqslant G$ can actually appear in the role of a stabilizer. For the case of torsion-free nilpotent group acting on the interval (open, semi-open, or compact), the following definition actually provides a criterion for actions by homeomorphisms:

\begin{definition}\label{def stab subg} Let $K$ be a subgroup of a finitely generated torsion-free nilpotent group $G$. Let $c\in K$ be a non-trivial central element of~$G$. We say that $K$ is a \emph{stabilizer subgroup} of $c$ if it satisfies the following conditions:
\begin{enumerate}
\item\label{i:H} $K/H\simeq \Z$ for some $H \triangleleft K$ that satisfies $H\cap \langle c\rangle=\{e\}$.
\item\label{i:chain} There exists a finite subnormal ascending chain 
\begin{equation}\label{eq:K-chain}
K=K_n\triangleleft K_{n-1}\triangleleft \cdots \triangleleft K_0=G,
\end{equation}
such that $K_i/K_{i+1}\simeq \Z$ for all $i\in \{0,\ldots,n-1\}$.
\end{enumerate}

  We denote the set of all stabilizer subgroups of an element $c\in Z(G)\smallsetminus \{e\}$ as $\StSub(G,c).$
\end{definition}

\begin{remark} It is straightforward to verify that the center of a torsion-free nilpotent group always satisfies the above definition. Indeed, this is a consequence of the fact that, in torsion-free nilpotent groups, the factors of the ascending central series are torsion-free (see, for example, \cite{CMZ}).
\end{remark}

\begin{remark}\label{rq:Zr}
The part~\ref{i:chain} of Definition~\ref{def stab subg} can be replaced by the following equivalent condition: there exists a chain 
\begin{equation}\label{eq:Kp-chain}
K=K_{(n')}\triangleleft K_{(n'-1)}\triangleleft \cdots \triangleleft K_{(0)}=G,
\end{equation}
such that for all $i$ one has $K_{(i)}/K_{(i+1)}\simeq \Z^{r_i}$ for some $r_i\in \N$. Indeed, for every $i$ denote by $\psi_{i,1}:K_{(i)}/K_{(i+1)} \to \Z^{r_i}$ the corresponding isomorphism, by $\psi_{i,2}:K_{(i)}\to K_{(i)}/ K_{(i+1)}$ the natural projection, and let $\psi_i:=\psi_{i,2}\circ \psi_{i,1}$. Now, take a chain of coordinate subgroups $0<\Z<\Z^2<\dots<\Z^{r_i}$ (with $\Z$ as every consecutive quotient), and replace every $K_{(i+1)}\triangleleft K_{(i)}$ in the chain~\eqref{eq:Kp-chain} by 
\[
K_{(i+1)} = \psi_i^{-1}(0) \triangleleft \psi_i^{-1}(\Z) \triangleleft \psi_i^{-1}(\Z^2) \triangleleft \dots \triangleleft \psi_i^{-1}(\Z^{r_i})=K_{(i)},
\]
with the quotient on every step being isomorphic to~$\Z$.
\end{remark}

\begin{prop}\label{p:stab-subgr}
A subgroup $K$ is a stabilizer subgroup for $c\in Z(G), \, c\neq e$, if and only if there exists an action of $G$ by homeomorphisms of $M=[0,1]$ and a support interval~$J$ for~$c$, such that $\Stab(J)=K$. 
\end{prop}

\begin{proof} 
Let us first prove the realisation part. Namely, assume that an ascending chain~\eqref{eq:K-chain} is given. We then recursively construct actions of subgroups~$K_j$ on some intervals~$J_{j}$, going along this chain all the way to the $j=0$ (when the action of $G$ is constructed). We start with $j=n$, taking an interval $J_n$ and its diffeomorphism $g_n$ with no fixed points inside. This diffeomorphism induces an action of the group $\Z$ on $J_n$ by $i\mapsto g_n^i$, that we extend to the action group $K=K_n$ by pre-composing it with the passage to the quotient $K_n \to K_n/H \simeq \Z$. As $\langle c \rangle \cap H =\{e\}$ by assumption, the element $c$ acts by some nonzero power of the diffeomorphism~$g_n$, and hence $I_c:=J_n$ is the support interval of the action of~$c$.

Now, assume that the action of some $K_{j+1}$ on $J_{j+1}$ is already constructed. Fix $g_j\in K_j$ that projects to~$1$ in the quotient $K_j/K_{j+1}\simeq \Z$, take interval $J_j\supset J_{j+1}$ and define the action of~$g_j$ on~$J_j$ in such a way that $J_{j+1}$ is the fundamental domain of~$g_j$ (see Fig.~\ref{f:extending-K}). 
\begin{figure}[!h!]
{\center
\includegraphics{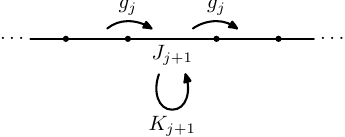} 

}
\caption{Extending the action of $K_{j+1}$, acting on $J_{j+1}$}\label{f:extending-K}
\end{figure}

Now, extend the action of $K_{j+1}$ on $J_j$: for every $i$ and any $h\in K_{j+1}$, define the restriction of $h$ on $g_j^i(J_{j+1})$ by
\begin{equation}\label{eq:extending}
h(g_j^i x) := g_j^i (\underbrace{(g_j^{-i} h g_j^i)}_{\in K_{j+1}} (x)), \quad x\in J_{j+1}.
\end{equation}
The action of $K_{j+1}$ then preserves each of the images of the fundamental domain $g_j^i(J_{j+1})$; gluing them together, one gets a continuous action of~$K_{j+1}$ on~$J_{j}$. Finally, every $g\in K_j$ can be represented as $g_j^i h$ for some $h\in K_{j+1}$ and $i\in\Z$, thus defining the action of all the subgroup $K_j$ on~$J_j$; it is easy to see from the definition~\eqref{eq:extending} that this is indeed an action. Also, by induction we see that the stabiliser of $I_c=J_n$ for this action consists only of the group~$K_n$ (if $i\neq 0$, then $g_j^i J_{j+1}$ is disjoint from $J_{j+1}$).

Proceeding until $j=0$, we define the action of the full group~$G$, for which $J_n=I_c$ is the support interval of the element~$c$, and $K_n=\Stab(I_c)$.

To proceed in the other direction, recall first the following general statement due to Plante~\cite{plante}. If a nilpotent group $G$ acts on the real line (or on the open interval) by orientation-preserving homeomorphisms, then there exists a Radon (infinite) invariant measure~$\nu$ of this action. This measure allows to define the \emph{translation number} $\tau(g)$ of each element $g\in G$, that is a group homomorphism from~$G$ to~$\R$. Finally, if the action has no global fixed points, then an element $g\in G$ has a fixed point if and only if its translation number vanishes. We refer the reader to~\cite{navas-book} for the details (see Theorem~2.2.39 and Section~2.2.5 therein).

Assume now that we are given an action of $G$ on $[0,1]$ and a support interval $J=I_c$ of~$c$, and let us prove that its stabiliser $\Stab(J)$ satisfies the definition of the stabiliser subgroup. Consider first the action of the subgroup $K=\Stab(I_c)$ on the open interval~$I_c^{\circ}$. The element~$c$ has no fixed points inside $I_c$, hence there is a translation number homomorphism $\tau_{I_c}:K\to \R$ that does not send $c$ to zero. Its image $\tau_{I_c}(K)\subset \R$ is a finitely generated abelian torsion-free group, hence isomorphic to some~$\Z^{l}$. In $\Z^l$, for every nonzero element $v_c$ there exists a subgroup~$\tH$ (actually, one of the coordinate hyperplanes $\Z^{l-1}$), intersecting trivially~$\langle v_c \rangle$ and such that $\Z^l / \tH \simeq \Z$. Hence, such a subgroup $\tH$, intersecting trivially $\langle \tau_{I_c}(c) \rangle$, exists in $\tau_{I_c}(K)$. Taking $H:=\tau_{I_c}^{-1}(\tH)$, we obtain the subgroup that satisfies the condition~\ref{i:H}.

Now, to check the ascending chain condition~\ref{i:chain}, we use Remark~\ref{rq:Zr} and construct the chain~\eqref{eq:Kp-chain}. We construct it, starting from the right (that is, from the group~$K_{(0)}=G$). If some $K_{(j)}$ is already constructed, consider the orbit $K_{(j)}(x)$ of any point $x\in I_c$ and let 
\[
J_j:= [\inf K_{(j)}(x), \sup K_{(j)}(x)]\supset I_c
\]
be a $K_{(j)}$-invariant interval, inside which the action of $K_{(j)}$ has no fixed points. Again, consider the associated translation number $\tau_{J_j}:K_{(j)}\to \R$. Its image is isomorphic to some~$\Z^{r_j}$; now, it suffices to take $K_{(j+1)}:=\Ker \tau_{J_j}$, so that $K_{(j)}/K_{(j+1)} \simeq \Im \tau_{J_j}\simeq \Z^{r_j}$. Repeating these steps, we get a process that ends when we reach the equality $J_j=I_c$, thus implying that $K_{(j)}=\Stab (I_c)$.
\end{proof}

\subsection{Growth of subgroups and Schreier graphs}

Given a group $G$ and a finite set $\mF$ of its generators, we consider the corresponding metric $d_{G;\mF}$ on~$G$ (distance in the Cayley graph, where each $g$ is joined by edges to all $fg$ with $f\in \mF$), and hence the corresponding radius $n$ balls 
\[
B_n^{G;\mF}(g):=\{h\mid d_{G;\mF}(g,h)\le n\}, \quad  g\in G, \quad n\in \N.
\]
This metric allows to define the \emph{growth function}
\[
\GrFn_{G;\mF}(n) = \# B_n^{G;\mF}(g);
\]
using another system $\mF'$ of generators, or passing to a finite-index subgroup, leads to a quasi-isometric metric. Though, it is easy to see that the new growth function $\GrFn'$ is equivalent to the old one $\GrFn$ in the sense of an equivalence relation given by 
\begin{equation}\label{eq:phi-prime}
\GrFn'\sim \GrFn \quad \Leftrightarrow \quad  \GrFn' \preccurlyeq \GrFn \,  \text{ and } \, \GrFn \preccurlyeq \GrFn',
\end{equation}
where
\[
\GrFn' \preccurlyeq \GrFn \Leftrightarrow \quad  \exists C: \quad \forall n\in \N \quad   \GrFn'(n) \le C \GrFn (Cn).
\]
Thus usually the growth function is considered up to this equivalence (given by~\eqref{eq:phi-prime}); one class of such equivalence is the class of polynomial growth of degree~$d$:
\begin{definition}
The group $G$ has a \emph{polynomial growth of degree}~$d$ if the limit 
\begin{equation}\label{eq:def-growth}
\growth(G):=\lim_{n\to \infty} \frac{\log \GrFn_{G;\mF}(n)}{\log n}
\end{equation}
exists and is equal to~$d$. 
\end{definition}

Two notions that we will need in this work is the \emph{relative} growth of a subgroup $H<G$, and the one of the growth of the corresponding Schreier graph~$G/H$. Namely, in the same way as the growth function $\GrFn_G$ counts all the points in a radius $n$ ball of the group~$G$, we can count the number of elements of the subgroup $H$ that are there. For the nilpotent groups, we consider the corresponding degree of polynomial growth:
\begin{definition}
The \emph{relative growth function} for a subgroup $H<G$ is defined by 
\[
\GrFn_H^G(n):=\# (B^{G;\mF}_n(e) \cap H)
\]
The subgroup $H<G$ has a \emph{polynomial relative growth of degree}~$d$ if the limit 
\begin{equation}\label{eq:def-rel-gr}
\growth_G(H):=\lim_{n\to \infty} \frac{\log \GrFn_H^G(n)}{\log n}
\end{equation}
exists and is equal to~$d$. 
\end{definition}
In the same way, for a given subgroup $H<G$ the quotient $G/H$ becomes a metric space (with the distance $d_{G/H;\mF}$ coming from the Schreier graph: the class $gH$ is joined by an edge to $fgH$ for all $f\in \mF$), so the corresponding balls $B_n^{G/H;\mF}(g)$ and the growth function 
\[
\GrFn_{G/H;\mF}(n):=\# B^{G/H;\mF}_n(e)
\]
are defined as before. Also, in the same way as before, we say that the \emph{polynomial growth rate} of $G/H$ is equal to $d$ if the limit
\begin{equation}\label{eq:def-growth-G/H}
\growth(G/H):=\lim_{n\to \infty} \frac{\log \GrFn_{G/H;\mF}(n)}{\log n}
\end{equation}
exists and is equal to~$d$.

From now on, we will assume that the system $\mF$ of generators of $G$ is fixed, and will write simply $B_n^G(g)$ and not $B_n^{G,\mF}(g)$.

\subsection{Nilpotent groups and Bass--Guivarc'h formula}

Given a group $G$, one considers the descending chain of commutators 
\[
G= G_1 \vartriangleright G_2 \vartriangleright \dots
\]
that is defined by 
\[
G_1:=G, \quad G_{j+1}:=[G,G_{j}], \quad j=1,2,\dots.
\]
\begin{definition}
The group $G$ is called \emph{nilpotent}, if for some $m$ one has $G_{m+1}=\{e\}$ (in other words, the descending chain of commutators reaches the trivial subgroup after a finite number of steps). 
\end{definition}

It is known (see~\cite[5.1.11]{robinson},~\cite[Theorem~2.52]{hall}) that for a group~$G$, one has the inclusions
\begin{equation}\label{eq:Gi-Gj}
[G_i,G_j]\le G_{i+j}, \quad i,j=1,2,\dots.
\end{equation}

Also, the following statement holds:
\begin{lemma}\label{l:on-centre}
Assume that the group $G$ is nilpotent, and $N\triangleleft G$ some normal subgroup, $N\neq \{e\}$. Then $N$ intersect non-trivially the center of the group, $N\cap Z(G)\neq\{e\}$. 
\end{lemma}
\begin{proof}
Note that for every $g\in G$, if $g\notin Z(G)$, there exists $g'\in G$ such that $[g,g']\neq e$. Hence, for every $g\neq e$ there exists a chain of commutators
\begin{equation}\label{eq:g-chain}
g_0=g, \quad g_{i+1}=[g_i,g'_i], \quad g'_i\in G,
\end{equation}
that ends with some $g_l\in Z(G) \setminus \{e\}$. (Such a chain ends in a finite number of steps due to the nilpotence of the group.)

On the other hand, for any $g\in N$ and $g'\in G$ one immediately gets 
\[
[g,g']=g \cdot (g' g^{-1} (g')^{-1})\in N.
\]
Thus, taking $g\in N\setminus \{e\}$, we see that all the chain of commutators given by~\eqref{eq:g-chain} consists of non-trivial elements, belonging to~$N$; in particular, this applies to its last element $g_l\in Z(G)$.
\end{proof}

Gromov's celebrated theorem on groups of polynomial growth states that finitely generated groups of at most polynomial growth are precisely the virtually nilpotent ones. Moreover, in this case, the Bass--Guivarc'h formula asserts that the growth function satisfies
\[
\GrFn_G(n) \asymp n^d, 
\]
and the growth rate $d=\growth(G)$ can be described explicitly in terms of the group~$G$. Here $\asymp$ is the notation for two sequences being comparable: by definition, $a_n \asymp b_n$ for two sequences $a_n,b_n$ of positive numbers if there exists a constant $C$ such that 
\[
\frac{1}{C} \le \frac{a_n}{b_n} \le C
\]
for all sufficiently large~$n$.
To state the Bass--Guivarc'h formula, giving the growth rate $\growth(G)$ of a finitely generated nilpotent group~$G$, let $d^G_j$ be the rank of the quotient $G_{j}/G_{j+1}$:
\begin{equation}\label{eq:def-dG}
G_{j}/G_{j+1} \simeq \Z^{d^G_j} \oplus T_j, \quad |T_j|<\infty.
\end{equation}
We will also denote the quotient map in~\eqref{eq:def-dG} by $\varphi_j:G_j\to G_j/G_{j+1}$, and by $\varphi^T_j$ its post-composition with the quotient by the torsion subgroup~$T_j$, 
\begin{equation}\label{eq:def-phi-T}
\varphi^T_j : G_j \to \Z^{d^G_j}.
\end{equation}

The statement of the classical Bass--Guivarc'h formula, giving the growth of the group $G$ in terms of the ranks $d^G_j$, is the following. 

\begin{theorem}[Bass--Guivarc'h formula,~{\cite{bass, guivarch}}]
The growth of a nilpotent group $G$ is given by
\[
\growth (G) =\sum_j j d^G_j.
\]
Moreover, one has $\GrFn_G(n) \asymp n^{\growth(G)}$.
\end{theorem}

\section{Main results}\label{s:main}

To present and state our results, let $G$ be a finitely generated, torsion-free nilpotent group. (We will repeat this assumption in the statements of the main results to avoid the risk of misunderstanding, but otherwise, from now on and unless explicitly stated otherwise, this is a standing assumption.) Our goal is to find the critical regularity for actions of $G$ on one-dimensional manifolds; more precisely, we'll be looking for an algebraic characterisation for it.

\subsection{Classes of actions associated to a given central element}

We start by considering an intermediate question. Given a nontrivial central element $c\in Z(G)\setminus \{e\}$, we consider the class of actions where only the cyclic group $\langle c \rangle$, generated by it, is required to act faithfully. For the actions on $[0,1]$ (or $[0,1)$), this condition is equivalent to a simpler one: it suffices to require only that~$c$ acts non-trivially. We also consider the subclass of actions on $[0,1]$ where the action is required to be tangent to the identity at the endpoints.

\begin{definition} For a one-dimensional manifold $M$, a group $G$ and a non-trivial central element $c\in G$, define
\[
\Class_{M}(G,c):=\left\{\phi : G \to \Diff^1_+(M)) \mid \forall n=1,2,\dots \quad \phi( c^n) \neq \id \right\}
\]
and the modification of this class for $M=[0,1]$ with an additional restriction of derivative at the endpoints: 
\[
\Class_{[0,1]}^0(G,c):= \left\{\phi \in \Class_{M}(G,c) \mid D\phi(g)(0) =D\phi(g)(1) = 1 \quad \forall g\in G \right\}.
\]
Let
\[
\Crit_M(G,c) := \Crit \Class_{M}(G,c)
\]
and 
\[
\Crit^0_{[0,1]}(G,c) := \Crit (\Class_{[0,1]}^0(G,c))
\]
be the corresponding critical regularities.
\end{definition}

Then, we have an explicit formula for the corresponding critical regularity. Namely, the following propositions provide respectively upper and lower bound for the regularity:

\begin{prop}\label{thm upper bound}
Let $\phi:G\to \Diff_+^1(M)$,
where $M$ is one of $[0,1]$, $(0,1]$ or $\Sc$, and let $c\in G$ be a non-trivial central element. 
Assume that $I_c$ is a support interval of $c$ that satisfies $\growth(G/\Stab(I_c)) >0$. Then, the action $\phi$ cannot be
of class $C^{1+\alpha}$ for any $\alpha> [\growth(G/\Stab(I_c))]^{-1}$. 
\end{prop}

\begin{prop}\label{p:lower-K}
Let $G$ be a torsion-free finitely generated nilpotent group, $c\in Z(G)$ a non-trivial central element, and let $K<G$ be a stabilizer subgroup for~$c$. Assume that $\alpha<[\growth(G/K)]^{-1}$. Then there exists an action $\phi\in \Class_{[0,1]}(G;c)$ of class $C^{1+\alpha}$ and a support interval $I_c$ of~$c$ for this action, such that~$K=\Stab(I_c)$.
\end{prop}

Together, they imply our first main theorem:

\begin{thmx}\label{t:A}
Let $G$ be a torsion-free finitely generated nilpotent group, $c\in Z(G)$ a non-trivial central element, and $M$ one of $[0,1]$, $(0,1]$ or~$\Sc$. Then
\[
\Crit \Class_M(G,c) = \Crit \Class_{[0,1]}^0 (G,c) = 
1+ \left[ \min_{K\in \StSub(G,c)}  \growth (G/ K) \right]^{-1}.
\]
\end{thmx}

Note that for an action of a nilpotent group, it suffices to check its faithfulness only for the central elements. Indeed, the kernel of a homomorphism is a normal subgroup, and by Lemma~\ref{l:on-centre} we know that a non-trivial normal subgroup of a nilpotent group~$G$ intersects its center~$Z(G)$ non-trivially. In particular, for nilpotent groups with a \emph{cyclic} center, Theorem~\ref{t:A} already allows to find the critical regularity of faithful actions:
\begin{corollary}\label{rk:cyclic}
Let $G$ be a torsion-free finitely generated nilpotent group with a cyclic center, and $c\in Z(G)$ a non-trivial central element. Then
\[
\textup{Crit}_{[0,1]}(G)=1+ \left[ \min_{K\in \StSub(G,c)}  \growth (G/ K) \right]^{-1}.
\]
\end{corollary}

After the first version of this paper was prepared, from our discussions with Andrés Navas resulted that the critical regularity cannot be attained, too:

\begin{prop}\label{p:critical}
In the assumptions of Proposition~\ref{thm upper bound}, the action $\phi$ also cannot have the regularity $C^{1+\alpha}$ with $\alpha= [\growth(G/\textup{Stab}(I_c))]^{-1}$. In particular, in the assumptions of Theorem~\ref{t:A}, the action cannot be of regularity exactly $\Crit\Class(G,c)$.
\end{prop}

\subsection{General critical regularity}

\begin{thmx}[Critical regularity]\label{thm crit interval} 
Let $G$ be a finitely generated, torsion-free non-abelian nilpotent group. Then, 
\begin{equation}\label{eq:th-A}
\Crit_{[0,1]}(G) = 1 + \left[ \max_{c\in Z(G)\setminus \{e\}} \min \{\growth(G/K) \, : \,  K \in \StSub(G,c)\} \right]^{-1}
\end{equation}
\end{thmx}

\begin{remark}\label{rmk abelian} 
J.~Plante and W.~Thurston have shown~\cite{pt} (see Theorem~4.5 therein) that an action of nilpotent non-abelian group on $[0,1]$ cannot be of regularity~$C^2$; hence the right hand side of~\eqref{eq:th-A} should be finite (and not exceeding~$1$). Let us check that it is indeed the case.

Indeed, one has $\growth(G/K_c)=0$ for some $c\in Z(G)\setminus \{e\}$ and a stabilizer subgroup $K_c\in \StSub(G,c)$ only if the ascending chain in the part~\ref{i:chain} of Definition~\ref{def stab subg} is trivial, $K_c=G$, and thus due to the part~\ref{i:H} for some $H_c\triangleleft G$ one has $G/H_c\simeq \Z$, and $H_c \cap \langle c \rangle = \{e\}$.

Now, if such $H_c$ exists for all $c\in Z(G)\setminus \{e\}$, then the homomorphism $G\mapsto \prod_{c\in Z(G) \setminus \{e\}} (G/H_c)$ is injective due to Lemma~\ref{l:on-centre}, as it is injective on $Z(G)$, and ranges in an abelian group. This would imply that $G$ is abelian, thus providing a contradiction.
 \end{remark}

\begin{thmx}[Half-open interval]\label{t:half-open}
Let $G$ be like in Theorem~\ref{thm crit interval}. Then
\[
 \Crit_{(0,1]}(G)=  \Crit_{[0,1]}(G)
\]
\end{thmx}

\begin{thmx}[Circle]\label{t:circle}
Let $G$ be like in Theorem~\ref{thm crit interval}. Then
\[
 \Crit_{\Sc}(G)=  \Crit_{[0,1]}(G)
\]
\end{thmx}

\begin{remark}\label{rq:circle-torsion}
Theorem~\ref{t:circle} admits a generalisation to groups with torsion that admit faithful actions on the circle. Namely, the set of torsion elements $T$ in such a group $G$ is a finite cyclic subgroup of the center $Z(G)$, the quotient group $G/T$ is a torsion-free nilpotent group, and 
\begin{equation}\label{eq:G/T}
 \Crit_{\Sc}(G)=  \Crit_{[0,1]}(G/T).
\end{equation}
\end{remark}

\subsection{Generalisations of the Bass--Guivarc'h formula}\label{s:gen-Bass}

We generalise the classical Bass-Guivarc'h formula to the case of Schreier graphs, as well as to the relative growth of subgroups. Namely, let a subgroup $H<G$ of a nilpotent finitely generated group $G$ be given. Denote
\[
H_{j}:=H\cap G_j.
\]
Consider the rank~$d^{H;G}_j$ of the image of $H_{j}$ under the map $\varphi_j: G_j\to G_j/G_{j+1}$:
\[
\varphi_j (H_{j}) \simeq \Z^{d^{H;G}_j} \oplus T'_j, \quad |T'_j|<\infty;
\]
equivalently, this is the rank of the image $\varphi^T_j(H_j)<\Z^{d^{G}_j}$, where $\varphi^T_j$ is given by~\eqref{eq:def-phi-T}.
Note that these ranks differ from the ones that one would obtain when applying Bass-Guivarc'h formula for the growth of $H$ as an abstract nilpotent group. For instance, for the Heisenberg group $G=N_3$ and $H=Z(G)=\langle c\rangle\simeq \Z$, one gets $d^{H;G}_1=0$ and $d^{H;G}_2=1$, instead of $d^{H}_1=1$ and $d^{H}_2=0$.

We then have the following result.
\begin{thmx}\label{t.B-G}
The relative growth of a subgroup $H<G$ is given by
\[
\growth_G(H) = \sum_j  j d^{H;G}_j =: \grH,  
\]
and the growth of the associated Screier graph is given by 
\[
\growth(G/H) = \sum_j j (d^G_j - d^{H;G}_j) = D_G-\grH.
\]
Moreover, both these growth functions are equivalent to polynomial, in the sense that 
\begin{equation}\label{eq:asymp-H}
\# (B_n^G(e)\cap H) \asymp n^{\grH} 
\end{equation}
and 
\begin{equation}\label{eq:asymp-G/H}
\# (B_n^{G/H}(e)) \asymp n^{D_G-\grH}.
\end{equation}
\end{thmx}

\section{Proof of Theorem~\ref{t.B-G}}\label{s:proof-E}

\subsection{Main proposition: comparing $\varphi^T_j$-images with the balls in $\Z^{d_j^G}$.}\label{s:varphi-j}

In order to prove Theorem~\ref{t.B-G}, we will be establishing directly the asymptotics~\eqref{eq:asymp-H} and~\eqref{eq:asymp-G/H}. Note first that it actually suffices to establish~\eqref{eq:asymp-H}.
\begin{prop}\label{p:G/H}
The asymptotics~\eqref{eq:asymp-H} directly implies~\eqref{eq:asymp-G/H}. 
\end{prop}
To prove it, notice first that we have the following estimates: 
\begin{lemma}
The following inequalities hold for all~$n$:
\begin{equation}\label{eq:G-lower}
\# B_{2n}^G(e) \ge \# B_n^{G/H}(e) \cdot \# B_{n}^{H;G}(e),
\end{equation}
\begin{equation}\label{eq:G-upper}
\# B_n^G(e) \le \# B_n^{G/H}(e) \cdot \# B_{2n}^{H;G}(e).
\end{equation}
\end{lemma}
\begin{proof}
Indeed, for~\eqref{eq:G-lower}, take a set $X_n\subset B_n^G(e)$ that contains exactly one preimage of each element of~$B_n^{G/H}(e)$ under the projection from $G$ to $G/H$. Now, all the products $gh$ for $g\in X_n$, $h\in B_n^{H;G}(e)$, are pairwise different, as the left classes $gH$ for $g\in X_n$ are pairwise disjoint. Now, all such products belong to $B_{2n}^G(e)$, and this implies~\eqref{eq:G-lower}. 

On the other hand, if two elements $g,g'\in B_n^G(e)$ project to the same vertex of the Schreier graph~$G/H$, then $g'=gh$ for 
\[
h=g^{-1} g', \quad h\in H,
\]
and hence $h\in B_{2n}^{H;G}(e)$. The preimage of every vertex in~$B_n^{G/H}(e)$ under the projection from $G$ to $G/H$ consists of at most $\# B_{2n}^{H;G}(e)$ elements, thus implying~\eqref{eq:G-upper}.
\end{proof}
\begin{proof}[Proof of Proposition~\ref{p:G/H}]
Combining~\eqref{eq:G-lower} and~\eqref{eq:G-upper}, we get 
\begin{equation}\label{eq:G/H-estimate}
\frac{\# B_n^G(e)}{\# B_{2n}^{H;G}(e)} \le B_n^{G/H}(e) \le \frac{\# B_{2n}^G(e)}{\# B_n^{H;G}(e)}.
\end{equation}
As $\# B_n^G(e)\asymp n^{D_G}$ and $\# B_n^{H;G}(e)\asymp n^{\grH}$ by assumption of~\eqref{eq:asymp-H}, both upper and lower estimates in~\eqref{eq:G/H-estimate} are $\asymp$-equivalent to $n^{D_G-\grH}$, thus completing the deduction of~\eqref{eq:asymp-G/H}.
\end{proof}

Now, the asymptotics~\eqref{eq:asymp-H}, applied to the subgroup~$H_j$ instead of~$H$, states that 
\begin{equation}\label{eq:growth-H-j}
\# (B_n^G(e)\cap H_j) \asymp n^{D_{H_j;G}}, \quad D_{H_j;G} = \sum_{j'=1}^m j' d^{H_j;G}_{j'} = \sum_{j'=j}^m j' d^{H;G}_{j'}.
\end{equation}

The asymptotics~\eqref{eq:growth-H-j}, in its turn, can be proven by backwards induction on~$j$. Namely, recall that the map $\varphi^T_j:G_j\to\Z^{d_j^G}$ is the composition of the quotient map by~$G_{j+1}$ and by the torsion subgroup~$T_j$. Then, we have the following statement.
\begin{lemma}\label{l:backwards-volume}
Assume that for every $j=1,\dots,m$ one has 
\begin{equation}\label{eq:H-j-ball-T}
\# \varphi^T_j (B^G_n(e) \cap H_j) \asymp n^{j d^{H;G}_j}
\end{equation}
Then for every $j=1,\dots,m$ the asymptotics~\eqref{eq:growth-H-j} holds. 
\end{lemma}
Note that for $j=1$ the asymptotics~\eqref{eq:growth-H-j} becomes the desired~\eqref{eq:asymp-H}, hence establishing it implies Theorem~\ref{t.B-G}.

\begin{proof}[Proof of Lemma~\ref{l:backwards-volume}]
We will now prove~\eqref{eq:growth-H-j} by backwards induction on~$j$. Namely, for $j=m+1$, we have $H_{j}=H_{m+1}=\{e\}$, and as $D_{H_{m+1};G}=0$, 
the statement holds automatically. 

Next, assume that for some $t$ the asymptotics~\eqref{eq:growth-H-j} holds for $j=t+1$, and let us establish it for $j=t$. Note that 
\[
1\le \frac{\# \left( \varphi_j (B^G_n(e) \cap H_j)\right)}{\# \left( \varphi^T_j (B^G_n(e) \cap H_j)\right)} \le \# T_j,
\]
hence the asymptotics~\eqref{eq:H-j-ball-T} implies 
\begin{equation}\label{eq:H-j-ball}
\# \left( \varphi_j (B^G_n(e) \cap H_j)\right) \asymp n^{j d^{H;G}_j}.
\end{equation}
Now, in the same way as~\eqref{eq:G-lower},~\eqref{eq:G-upper}, as $\varphi(H_j) \simeq H_j/H_{j+1}$, we get
\[
\# (B^G_n(e) \cap H_j) \le \# \left(\varphi_j (B^G_n(e) \cap H_j)\right) \cdot \# (B^G_{2n}(e) \cap H_{j+1}),
\]
\[
\# (B^G_{2n}(e) \cap H_j) \ge \# \left( \varphi_j (B^G_n(e) \cap H_j)\right) \cdot \# (B^G_{n}(e) \cap H_{j+1}).
\]
Joining~\eqref{eq:H-j-ball} with that $\# (B^G_{n}(e) \cap H_{j+1}) \asymp n^{D_{H_{j+1};G}}$ by the induction assumption, we get 
\[
\# (B^G_n(e) \cap H_j) \asymp n^{D_{H_{j+1};G}} \cdot n^{n^{j d^{H;G}_j}} = n^{D_{H_{j+1};G} + j d^{H;G}_j} = n^{D_{H_j;G}}.
\]
This concludes the induction step and thus the proof.
\end{proof}

The following lemma, joining classical statements, can be seen as one of the steps of the proof of Bass-Guivarc'h Theorem. Roughly speaking, it says that $\varphi^T_j$-images of intersections $B_n^G(e)\cap G_j$ scale as $n^j$, that is, contain and are contained in balls of radii~$\const\cdot n^j$. We refer the reader to~\cite[Section 6]{bass} and to~\cite[Section~II]{guivarch}:

\begin{lemma}[Bass, Guivarc'h]
For every $j$ there exists a constant $C_j$ such that for all $n$
\begin{equation}\label{eq:j-powers-contains-G}
B_{n^j}^{\Z^{d^G_j}}(0) \subset \varphi^T_j(B_{C_j n}^G(e)\cap G_j)
\end{equation}
\begin{equation}\label{eq:j-powers-G}
\varphi_j^T(B_n^G(e) \cap G_j) \subset B_{C_j n^j}^{\Z^{d^G_j}}(0)
\end{equation}
\end{lemma}

To prove our Theorem~\ref{t.B-G}, we will need a stronger version of this statement, considering $\varphi^T_j$-images of intersections $B_n^G(e)\cap H_j$. Namely, the proposition below states that these images contain and are contained in spheres of radii $\const \cdot n^j$ in the rank $d^{H;G}_j$ lattice $\varphi^T_j(H_j)<\Z^{d^G_j}$.

\begin{prop}\label{p:j-powers-H}
For every $j$ there exists a constant $C_j$ such that for all $n$
\begin{equation}\label{eq:j-powers-contains-H}
B_{C_j^{-1} n^j}^{\Z^{d^G_j}}(0) \cap \varphi^T_j(H_j)  \subset \varphi^T_j(B_n^G(e) \cap H_j)
\end{equation}
\begin{equation}\label{eq:j-powers-H}
\varphi_j^T(B_n^G(e) \cap H_j) \subset B_{C_j n^j}^{\Z^{d^G_j}}(0)
\end{equation}
\end{prop}
The upper bound~\eqref{eq:j-powers-H} directly follows from~\eqref{eq:j-powers-G}. However, the lower bound~\eqref{eq:j-powers-contains-H} is stronger than the inclusion~\eqref{eq:j-powers-contains-G}. Indeed, the intersection of images $\varphi^T_j(H_j)\cap \varphi^T_j(B_n^G(e))$ is \emph{a priori} larger than the image of the intersection $\varphi^T_j(H_j \cap B_n^G(e))$. So it might happen that for too many points that can be both obtained as a $\varphi^T_j$-image of a point from the subgroup  $H_j$ and as an image of some points from the ball $B_{cn}(G_j)$, there would be no way to get them as a $\varphi^T_j$-image of a point from the subgroup~$H_j$ within such a ball. It seems also that there is no way to modify the commutator construction to get an element from the subgroup $H_j$ right away. For instance, if one tries to use elements of $H$ instead of those of $G$ in the commutators chain, the commutators might become trivial (for instance, for the case $G=N_3$ for the central subgroup $H=Z(N_3)<N_3=G$).

We will prove Proposition~\ref{p:j-powers-H} below; for completeness of the exposition, we will also present the proof of the upper bounds~\eqref{eq:j-powers-G}, as tools used to obtain it will be also used to obtain the lower bound~\eqref{eq:j-powers-contains-H}.

\subsection{Canonical form of an element}

We start by introducing a canonical form of an element of a nilpotent group; while doing so, we will also fix a particular system of generators of~$G$ with which we will be working from now on. Namely, take every $j=1,2,\dots$; from~\eqref{eq:def-dG} we have
\[
\varphi_j(G_j) = G_j/G_{j+1} \simeq T_j \oplus \Z^{d^G_j}.
\]
Fix a set $\wT_j\subset \varphi_j^{-1}(T_j) <G_j$, consisting of one $\varphi_j$-preimage of each of the elements of~$T_j$: 
\[
\forall t \in T_j \quad \exists! \, t'\in \wT_j : \quad \varphi_j(t')=t,
\]
taking $e\in \wT_j$ as the $\varphi_j$-preimage of $e\in T_j$. Also, fix elements $f_{1,j},\dots, f_{d^G_j,j}\in G_j$ such that their images under $\varphi^T_j$ form a base of 
$\Z^{d^G_j}= \varphi^T_j(G_j)$. Then, we have the following (the reader is also referred, for instance, to~\cite[Chapter~3]{CMZ} for the details) canonical form:

\begin{prop}[Canonical form]
Every element $g\in G$ can be written in a unique way in the form 
\begin{equation}\label{eq:canon}
\left(t_1 f_{1,1}^{a_{1,1}} \dots f_{d^G_1,1}^{a_{d^G_1,1}} \right) \left(t_2 f_{1,2}^{a_{1,2}} \dots f_{d^G_2,2}^{a_{d^G_2,2}} \right) \dots 
\left(t_m f_{1,m}^{a_{1,m}} \dots f_{d^G_m,m}^{a_{d^G_m,m}} \right),
\end{equation}
where $t_j \in \wT_j, \quad j=1,2,\dots,m$, and $a_{i,j}\in \Z$ are some integer powers.
\end{prop}

\begin{proof}
For any $j$, let us prove that elements $g\in G_j$ can be written in a unique way in the form~\eqref{eq:canon}, where only the brackets with numbers $j, j+1,\dots$ are used: 
\begin{equation}\label{eq:j-canon}
g= \underbrace{\left(t_j f_{1,j}^{a_{1,j}} \dots f_{d^G_j,j}^{a_{d^G_j,j}}\right)}_{\in G_j}  \dots 
\underbrace{\left(t_m f_{1,m}^{a_{1,m}} \dots f_{d^G_m,m}^{a_{d^G_m,m}} \right)}_{\in G_m}
\end{equation}
We will prove it by backward induction by $j$: for $j=m+1$, $G_j=\{e\}$ and the statement is trivial. Now, apply $\varphi^T_j$ to~\eqref{eq:j-canon}: in the right hand side only the image of the first bracket is nontrivial, and hence there are unique $a_{1,j},\dots, a_{d^G_j,j}\in \Z$, such that the images of the left and the right hand coincide. Fixing them, in the same way there is a unique $t_j\in \wT_j$ such that $\varphi_j$-images of the left and the right hand sides coincide. Then, it suffices to take the quotient $g'$, given by
\[
g=\left( t_j f_{1,j}^{a_{1,j}} \dots f_{d^G_j,j}^{a_{d^G_j,j}}\right) \cdot g';
\]
as $\varphi_j(g')=e$, one has $g'\in G_{j+1}$, and one can apply the induction assumption to it. The induction step is proven, and the statement for $j=1$ is exactly the statement of the proposition.
\end{proof}

We will now denote 
\[
\mF'_{(j)}:=\wT_j \cup \{f_{1,j},\dots, f_{d^G_j, j} \}, \quad \mF_{(j)}:=\mF'_{(j)} \cup \dots\cup \mF_{(m)}'. 
\]
Then for any $j$ the set $\mF_{(j)}$ is a generating set for~$G_j$.

\subsection{Upper bound: establishing inclusion~\eqref{eq:j-powers-G}}\label{s:upper-inclusion}

\begin{proof}
For every element $g\in B_n^G(e)$, let us write it as a product of generators in the canonical form~\eqref{eq:canon}. 
Actually, the following statement holds; it largely refers to~\cite{bass} (see Proposition~1 and Lemma~4 therein), but for the readers' convenience we present here its proof. 
\begin{lemma}\label{l:canon-length}
There exists a constant $C_A$ such that for every $n$ and any $g\in B_n^G(e)$, the canonical form~\eqref{eq:canon} satisfies the bound on the powers 
\begin{equation}\label{eq:sum-a-i-j}
\forall j \quad \sum_{i =1}^{d_j^G} |a_{i,j}| \le C_{A,j} \cdot n^j.
\end{equation}
\end{lemma}

This lemma immediately implies that once $g\in B_n^G(e)\cap G_j$, the image 
\[
\varphi_j^T(g) = (a_{1,j},\dots,a_{d_j^G,j}) \in \Z^{d_j^G}
\]
belongs to the ball of radius $C_{A,j} n^j$; hence, once it is proved, we get the desired inclusion~\eqref{eq:j-powers-G}.

The main idea of the proof of Lemma~\ref{l:canon-length} is the following: by definition, the element $g\in B_n^G(e)$ can be written as a product of at most $n$ generators from $\mF$ (as well as their inverses). Now, one starts to re-odrer these generators in the lexicographical order, bringing first the element of $\wT_1$, then all the powers of $f_{1,1}$, then all the powers of $f_{2,1}$, etc. Each time, to interchange two elements $a\in \mF'_{j_1}, \, b\in \mF_{j_2}, \, j_2>j_1$, that go in the wrong order, one uses the relation 
\begin{equation}\label{eq:exchange-ab}
ba=ab g_{a,b}, \quad \text{where } g_{a,b}:=[b^{-1},a^{-1}]\in G_{j_1+j_2},
\end{equation}
writing the commutator in the right hand side as a product of a generators in $\mF'_j, \quad j\ge j_1+j_2$. 
At the same time, when two elements $t_1,t_2\in \wT_j$ occur next to each other, they are merged using
\begin{equation}\label{eq:merge}
t_1 t_2 = t' g'_{t_1,t_2}, \quad \text{for some  }\, t'\in \wT_j, \quad g'_{t_1,t_2}\in G_{j+1}.
\end{equation}

As the relations~\eqref{eq:exchange-ab} are applied only for a finite set of possible $a,b\in \mF$, the number of generators used to write the commutators appearing in the right hand of~\eqref{eq:exchange-ab} side are uniformly bounded by some constant $C_{comm}$. In the same way, the number of generators used to write the correction terms $g$ in~\eqref{eq:merge} is also bounded by some constant~$C_T$. Finally, for the simplicity of the afterwards arguments, we will assume that these writings of $g_{a,b}$ and $g'_{t_1,t_2}$ are already their canonical forms.

Now, one can (as it was actually done in~\cite{bass}) keep track on how many generators of each level~$j$ might occur as a result of this process of interchanging. Interchanging $n$ generators from $\mF_1$ might create at most $\sim\const\cdot n^2$ generators from $\mF_2$, etc.. Though, accurate estimates here might be cumbersome (in~\cite{bass}, powers of idempotent operators are used for that), so we propose the following go-around (that will be also used to establish the inclusion~\eqref{eq:j-powers-contains-H}).

\begin{proof}[Proof of Lemma~\ref{l:canon-length}]
Fix a positive constant $A$. Now, for every given $n$, define the \emph{weight} of a generator $g\in \mF_j'$  as 
\[
\wt_{n}(g):=\frac{1}{A^{j-1}}\cdot n^{-j}. 
\]
Say that a pair $(g_1,g_2)$ is \emph{disordered} if we would have to interchange or merge them if they were next to each other in this order. Then, for a word $w=g_1 \dots g_k$ in generators from $\mF$ and their inverses, define its weight as
\begin{equation}\label{eq:weight}
\wt_{n} (w):= \sum_{i} \wt_{n}(g_i) + \sum_{i_1<i_2: \atop (g_{i_1},g_{i_2}) \, \text{ is disordered}} \wt_{n}(g_{i_1})\wt_{n}(g_{i_2}).
\end{equation}
Then, the following statement holds: 
\begin{lemma}\label{l:A}
We can choose the constant $A$ sufficiently large, so that for every $n$, if $w$ is a word of weight $\wt_{n}(w)\le 2$, and $w'$ a word that is obtained from $w$ by applying one of the operations~\eqref{eq:exchange-ab},~\eqref{eq:merge} to a disordered pair of its consecutive elements, then $\wt_{n}(w')\le \wt_{n} (w)$. 
\end{lemma}
\begin{proof}
Consider first the case when the interchange~\eqref{eq:exchange-ab} is applied, so that we pass from some word $w=w_1 ba w_2$ to $w'=w_1 ab g_{a,b} w_2$. Then we have 
\begin{equation}\label{eq:wt-difference}
\wt_{n}(w) - \wt_{n}(w') \ge \wt_{n}(a)\wt_{n}(b) - \wt_{n}(g_{a,b}) - \wt_{n}(g_{a,b})\cdot \wt_{n}(w), 
\end{equation}
where the first summand corresponds to the product $\wt_{n}(a)\wt_{n}(b)$ disappearing from the disorder sum, the second is due to $g_{a,b}$ appearing in the sum of weights of generators, and the third is an upper bound for its contribution to the disorders. For $a\in \mF'_{j_1}, b\in \mF'_{j_2}$, taking into account $\wt_{n}(w)\le 2$, the right hand side in~\eqref{eq:wt-difference} can be estimated as
\begin{multline*}
\wt_{n}(a)\wt_{n}(b) - \wt_{n}(g_{a,b}) - \wt_{n}(g_{a,b})\cdot \wt_{n}(w) \ge 
\\
\ge \frac{1}{A^{j_1+j_2-2}} n^{-j_1-j_2} - \frac{3C_{comm}}{A^{j_1+j_2-1}} n^{-j_1-j_2} = \frac{1}{A^{j_1+j_2-2}} n^{-j_1-j_2} \cdot \left(1-\frac{3C_{comm}}{A}\right) > 0
\end{multline*}
once $A> 3 C_{comm}$.

Now, when the lifts of torsion elements merging~\eqref{eq:merge} is implemented with some $t_1,t_2\in\wT_j$, replacing 
$w=w_1 t_1t_2 w_2$ by $w'=w_1 t'_{t_1,t_2} g'_{t_1,t_2} w_2$, we have 
\begin{equation}\label{eq:wt-merge-difference}
\wt_{n}(w) - \wt_{n}(w') \ge [\wt_{n}(t_1)+\wt_{n}(t_2) -\wt_{n}(t'_{t_1,t_2})] - \wt_{n}(g'_{t_1,t_2}) - \wt_{n}(g'_{t_1,t_2})\cdot \wt_{n}(w), 
\end{equation}
and  the right hand is bounded from below by
\[
\frac{1}{A^{j-1}} n^{-j} - \frac{3 C_T}{A^j} n^{-j-1} = \frac{1}{A^{j-1}} n^{-j} \cdot \left(1-\frac{3C_T}{nA} \right);
\]
again, the right hand side is guaranteed to be positive once~$A>3C_T$. 
\end{proof}

Now, Lemma~\ref{l:A} immediately implies the statement of Lemma~\ref{l:canon-length}. Indeed, for every $g\in B_n^G(e)$, the word $w_0$ of length~$n$ that writes it as a product of generators $g_i$ satisfies 
\[
\wt_{n}(w_0) \le \sum_{i} \wt_{n}(g_i) + \left(\sum_{i} \wt_{n}(g_i)\right)^2 \le 1+1^2=2.
\]
Let $w_k$ be the sequence of words that we obtain by the procedure of ``bubble sorting''. Then by induction we see that $\wt_{n}(w_k)$ is a decreasing sequence, and in particular, $\wt_{n}(w_k)\le 2$ for all~$k$.

Hence, the same upper bound for the weight holds for the canonical form (that we reach at the end of the chain $(w_k)$). In particular, for every $j$ the total weight of generators from $\mF'_j$ used in the canonical form doesn't exceed~$2$ (as it is a part of the full weight), and hence 
\[
\frac{1}{A^{j-1}} n^{-j}  \sum_{i=1}^{d^G_j} |a_{i,j}| \le 2,
\]
implying the desired~\eqref{eq:sum-a-i-j} with the constant~$C_{A,j}=2A^{j-1}$. 
\end{proof}

\subsection{Lower bound: establishing inclusion~\eqref{eq:j-powers-contains-H}}

Let us first establish~\eqref{eq:j-powers-contains-G}. The following lemma is a reformulated version of statements from~\cite[p.~609]{bass} and \cite[p.~344]{guivarch}; we include its proof for completeness.

\begin{lemma}\label{l:comm-power}
For any $a\in G_{j-1}$, $b\in G$ and any $k,l$ one has 
\begin{equation}\label{eq:ab-kl}
\varphi_j([a^k,b^l]) = (\varphi_j([a,b]))^{kl}.
\end{equation}
Also, if $\varphi_{j-1}(a)=\varphi_{j-1}(\ta)$, then $\varphi_j([a,b])=\varphi_j([\ta,b])$
\end{lemma}

\begin{proof}
Note that $[a,b]\in G_j$, hence for any element $g\in G$ we have $[[a,b],g]\in G_{j+1}$. Now, for any $g_1,g_2\in G$ we have 
\[
g_1 [a,b] g_2 = g_1 g_2 [a,b] \cdot h, \quad h=[[a,b],g_2]^{-1},
\]
hence the products $g_1 [a,b] g_2 = g_1 g_2 [a,b]$ simultaneously belong or not to $G_j$, and if they do, their images under $\varphi_j$ coincide, implying in this case
\[
\varphi_j(g_1 [a,b] g_2) = \varphi_j(g_1 g_2) \varphi_j([a,b]).
\]
As $ab=[a,b]ba$, for any $g_1,g_2\in G$ such that $g_1 ab g_2\in G_j$, we have 
\begin{equation}\label{eq:g1-ab-g2}
\varphi_j(g_1 ab g_2) = \varphi_j (g_1 [a,b]ba g_2) = \varphi_j (g_1 ba g_2) \varphi_j([a,b]).
\end{equation}

Applying~\eqref{eq:g1-ab-g2} to interchange consecutively $a$'s and $b$'s in $[a^k,b^l]=a^k b^l a^{-k} b^{-l}$,
we obtain the desired~\eqref{eq:ab-kl}. 

For the second part, the equality $\varphi_{j-1}(\ta)=\varphi_{j-1}(a)$ implies that $\ta=ah$ for some $h\in G_j$. Using this, we can write
\[
[\ta,b] = ahbh^{-1}a^{-1}b^{-1} =   a \, \underbrace{[h,b]}_{\in G_{j+1}} \, b a^{-1}b^{-1} = \underbrace{[a,[h,b]] \cdot [h,b]}_{\in G_{j+1}} \cdot [a,b],
\]
and the application of $\varphi_j$ concludes the proof.
\end{proof}

Now, to complete the proof of~\eqref{eq:j-powers-contains-G}, we will again proceed by induction on~$j$; we reproduce here the arguments 
from~\cite[Lemma~3]{bass} (see also~\cite[p.~344]{guivarch}). Recall that $f_{j,1},\dots, f_{j,d^{G}_j}$ are sent by $\varphi^T_j$ to the base elements of $\Z^{d_j^G}$. To prove the induction statement, it suffices to show that for each $i=1,\dots, d_j^{G}$ there exists a sufficiently large~$C_{j;i}$ such that for every $n$ we have 
\begin{equation}\label{eq:k-power}
\forall |k|\le n^j \quad   \varphi^T_j(f_{j,i}^k) \in \varphi^T_j(B_{C_{j;i} n}^G(e)).
\end{equation}
Indeed, once~\eqref{eq:k-power} is proven, for an arbitrary $a=(k_1,\dots, k_{d^G_j})\in \Z^{d_j^G}$ one can multiply the elements $g_i\in B_{C_{j;i} n}^G(e)$, such that $\varphi^T_j(g_i)= f_i^{k_i}$, and obtain $g\in B_{C'_{j} n}^G(e)$ with $\varphi^T_j(g)=a$, where  $C'_j=\sum_i C_{j;i}$.

Now, as $f_{j,i}\in G_j$, it has a representation of the form 
\[
f_{j,i} = [a_{1},b_{1}]\dots [a_{r},b_{r}]
\]
for some $a_1,\dots,a_r\in G_{j-1}, \quad b_1,\dots,b_r\in G$. Applying Lemma~\ref{l:comm-power}, we see that for every $|k'|\le n^{j-1}, 1\le l\le n$, the image of its $k'l$-th power can be written as
\[
\varphi^T_j(f_{j,i}^{k'l}) = \varphi^T_j([a_1^{k'},b_1^l] \dots [a_r^{k'},b_r^l]).
\]
Now, for each of the powers $a_s^{k'}$ due to the induction assumption (and the finiteness of the set $T_j$) there exists a word $w_s$ of length at most $\const n$ such that $\varphi_j(w_s)=\varphi_j(a_s^{k'})$. Hence,
\begin{equation}\label{eq:k'l}
\varphi^T_j(f_{j,i}^{k'l}) = \varphi^T_j([w_1,b_1^l] \dots [w_r,b_r^l]).
\end{equation}
Finally, writing $k=n q + r$ and multiplying the representations~\eqref{eq:k'l} for $k'=q, \, l=n$ and for $k'=r, l=1$, one gets a representation for $f_{i,j}^k$ of length at most $\const \cdot n$.
\end{proof}

Let us now establish~\eqref{eq:j-powers-contains-H}, thus completing the proof of Proposition~\ref{p:j-powers-H} and hence of Theorem~\ref{t.B-G}. We will start by establishing the following lemma.

\begin{lemma}\label{l:delta}
For any $\eps>0$ there exists $\delta>0$ such that for any $g\in G$ that can be written by a word $w$ of weight $\wt_n(w)<\delta$, one has $g\in B_{\eps n}^G(e)$. 
\end{lemma}
\begin{proof}
We will prove by backwards induction on $j$ this statement under an additional assumption that $g\in G_j$. As always, for $j=m+1$ there is nothing to prove. Now, let us establish the (backwards) induction step. Namely, take first some $\delta_1<1$. Assuming that $\wt_n(w)<\delta_1^j$ for a word $w$ defining an element $g\in G_j$, from~\eqref{eq:j-powers-contains-G} we can find $g'\in G_j$ such that $\varphi_j(g)=\varphi_j(g')$, and that $g'$ is a product of at most $C_j' \delta_1 n$~generators. Hence, once $\delta_1$ is chosen sufficiently small so that $C_j'\delta_1<1$, the weight $\wt_n(g')$ satisfies the upper bound 
\[
\wt_n(g') \le C_j' \delta_1 + (C_j' \delta_1)^2 < 2 C_j' \delta_1.
\]
Now, consider the composition quotient $g''=(g')^{-1}g$. As $\varphi_j(g)=\varphi_j(g')$, one has $g''\in G_{j+1}$. On the other hand, its weight satisfies 
\[
\wt_n(g'')\le \wt_n(g) + \wt_n(g') + \wt_n(g) \wt_n(g') < 3 \wt_n(g) + \wt(g'),
\]
so once $\wt_n(g)<\delta<\delta_1$, it does not exceed $3\delta + 2 C_j'\delta_1<(3+2C_j')\delta_1$. Now, as $g''\in G_{j+1}$, from the induction assumption there exists $\delta'$ such that if $\wt_n(g'')<\delta'$, then $g''$ can be written as a product of at most $\frac{\eps}{2} n$ generators. Fix such $\delta'$, and choose $\delta_1>0$ sufficiently small such that 
\[
(3+2C_j')\delta_1<\delta', \quad C_j'\delta_1<\min(\frac{\eps}{2},1),
\]
and take $\delta=\delta_1^j$. Then both $g'$ and $g''$ can be written as products of at most $\frac{\eps}{2}n$ generators, thus $g=g'g''$ belongs to $B_{\eps n}^G(e)$. This completes the (backwards) induction step, and for $j=1$ the statement becomes the conclusion of the lemma.
\end{proof}

\begin{remark}
The statement of Lemma~\ref{l:delta} can be reformulated in the following way, in a sense, reversing Lemma~\ref{l:canon-length}. There exist constants $c_{A,j}>0$, such that any $g\in G$, whose canonical the canonical form~\eqref{eq:canon} satisfies the bounds on the powers 
\begin{equation*}
\forall j \quad \sum_{i =1}^{d_j^G} |a_{i,j}| \le c_{A,j} \cdot n^j,
\end{equation*}
belongs to $B_n^G(e)$.
\end{remark}

\begin{proof}[Proof of Proposition~\ref{p:j-powers-H}]
To conclude the proof of the proposition, it suffices to establish the inclusion~\eqref{eq:j-powers-contains-H}. Indeed, consider the elements $g_1,\dots, g_{d^{H;G}_j}\in H_j$ such that their $\varphi^T_j$-images form a base of $\varphi^T_j(H_j)$. Then, one has $\wt_n(g_i)<L n^{-j}$ for some uniform constant $L$ (related to the lengths of their canonical forms). Hence, $\wt_n(g_i^{a_i})< 2La_i n^{-j}$ for all $a_i\in \Z$ such that $|a_i|<\frac{1}{L} n^j$. Applying Lemma~\ref{l:delta} to such powers and multiplying them, and we obtain a representation of length at most $n$ for some $g=g_1^{a_1}\dots g_{d^{H;G}_j}^{a_{d^{H;G}_j}}\in H_j$ with $\varphi^T_j(g)$ belonging to some radius $c n^j$-ball, intersected with~$\varphi^T_j(H_j)$. This (after adding a finite number of possibilities for the torsion elements) completes the proof of~\eqref{eq:j-powers-contains-H}, and hence of Proposition~\ref{p:j-powers-H} and of Theorem~\ref{t.B-G}.
\end{proof}

\section{Upper bounds: proof of Propositions~\ref{thm upper bound} and~\ref{p:critical}}\label{s:upper}
\subsection{Case of compact $M$}\label{s:upper-compact}

As it was already mentioned, the proof of Proposition~\ref{thm upper bound} will be based on the application of the Generalised Kopell's Lemma. Namely, let $G$ be a finitely generated nilpotent group, and let $\mF=\{f_1,\dots, f_k\}$ be a set of its generators, their inverses, and the neutral element~$e$. Assume that we are given an action of a nilpotent group $G$ on $M=[0,1]$ or $M=\Sc$, an interval $I_c\subset M$ is a support interval  for some central element $c\in Z(G)$, and the corresponding stabiliser subgroup is $K=\Stab(I_c)$. The quotient space $G/K$ is then naturally identified with the orbit $G(I_c)$. We are going to establish the following proposition.
\begin{prop}\label{l:alpha-conv}
In the assumptions above, for every $\alpha>[\growth(G/K)]^{-1}$ there exists a path 
\[
g_0=e, \quad g_{n}=f_{w_n}g_{n-1}, \quad n=1,2,\dots
\]
in the group $G$, such that
\begin{equation}\label{eq:g-sum}
\sum_n |g_n(I_c)|^{\alpha} < \infty.
\end{equation}
\end{prop}

Once established, Proposition~\ref{l:alpha-conv} implies Proposition~\ref{thm upper bound} immediately due to Generalised Kopell's Lemma (Proposition~\ref{p:Kopell}).

To establish Proposition~\ref{l:alpha-conv}, we will look for $g_n$ as a trajectory of a \emph{random process}, such that the \emph{expectation} of the sum~\eqref{eq:g-sum} is finite. Once such a process is constructed, almost every trajectory of such a process would satisfy the conclusion of Proposition~\ref{l:alpha-conv}. We will assume that $d=[\growth(G/K)]>1$, as the case $d=1$ is straightforward (in that case, $G/K\simeq \Z$, and it suffices iterate the generator of the quotient). 

For the case $G=\Z^d$, such an approach was already implemented in~\cite{DKN_acta}, with the process based on the Polya urn. However, this construction was heavily relying on the abelian nature of the group, and doesn't seem to generalise even to the Heisenberg group. A new way of constructing such a process was thus required. We are providing it in the proof of the following statement, that seems to be of interest of its own. 

\begin{prop}\label{p:process}
Let $G$ be any finitely generated group, $\mF$ the set of its generators, their inverses and~$e$. Then there exists a random process 
\[
g_0=e, \quad g_{n}=f_{w_n}g_{n-1}, \quad f_{w_n}\in\mF, \quad n=1,2,\dots,
\]
such that for every $n$ and any $g\in G$, one has 
\begin{equation}\label{eq:p-n-upper}
\Prob(g_n=g)\le \frac{1}{\# B^G_{n/4}(e)}.
\end{equation}
\end{prop}

Proposition~\ref{l:alpha-conv} follows from Proposition~\ref{p:process} and from the relative growth estimates for the nilpotent groups.
\begin{proof}[Deducing Proposition~\ref{l:alpha-conv} from Proposition~\ref{p:process}]

Note that for the case of the nilpotent group $G$, one has
\[
\#B_{n/4}^G(e) \ge C_1 n^{d_G}, 
\]
and for its subgroup $K$
\[
\#B_{2n}^G(e)\cap K \le C_2 n^{d_{K,G}};
\]
finally, the growth of $G/K$ is equal to $d=d_G-d_{K,G}$
(see Theorem~\ref{t.B-G}).

Take the random sequence~$g_n$ from the conclusion of Proposition~\ref{p:process}. As it was mentioned earlier, we are going to show that the 
expectation of the sum~\eqref{eq:g-sum} is finite. To do it, consider the sequence of (random) images $I_n:=g_n(I_c)$, and rewrite 
\[
\E \left(\sum_n |g_n(I_c)|^{\alpha} \right) = \sum_n  \E\left(|g_n(I_c)|^{\alpha}\right) = \sum_n  \E\left(|I_n|^{\alpha}\right).
\]
For any given $n$ the expectation in the right hand side can be rewritten as 
\begin{equation}\label{eq:E-I-n}
\E\left(|I_n|^{\alpha}\right) = \sum_{I'} |I'|^{\alpha} \cdot \Prob(I_n = I'),
\end{equation}
where the summation is over intervals $I'$ the $G$-orbit of~$I_c$. 

The behaviour of the probabilities in the right hand side can be estimated from Proposition~\ref{p:process}. Namely, for every $n$ and any interval $I'$ on the orbit of $I_c$, 
\[
I'=\overline{g}(I_c), \quad \overline{g}\in G,
\]
one has 
\begin{equation}\label{eq:I-Ip}
\Prob(I_n=I') = \Prob(\{ g_n=g \, \text{for some $g$ s.t. } \,  g(I_c)=\overline{g}(I_c)\}) = \sum_{h\in K} \Prob(g_n=\overline{g}h).
\end{equation}
If the event in~\eqref{eq:I-Ip} is non-empty, one can take $\overline{g}$ inside the ball $B_n^G(e)$, and thus restrict the sum in the right hand side to the radius $2n$ ball, that is, to $h\in B_{2n}^G(e) \cap K$. The upper bound~\eqref{eq:p-n-upper} then implies that for every $I'$ one has
\begin{equation}\label{eq:I-upper}
\Prob(I_n=I') \le  \frac{\# (B_{2n}^G(e) \cap K)}{\# B_{n/4}^G(e)} \le \frac{C_2 \cdot (2n)^{d_{K,G}}}{C_1 \cdot (n/4)^{d_G}} = \frac{C_3}{n^{d_G - d_{K,G}}} = \frac{C_3}{n^d}.
\end{equation}

Roughly speaking, the upper bound~\eqref{eq:I-upper} for the probabilities is similar to what would happen if the interval $I_n$ was uniformly chosen in the radius $n$ neighbourhood of $I_c$ in the corresponding Schreier graph. To conclude the proof, we will use this upper bound to estimate the expectation~\eqref{eq:E-I-n}. Namely, the following statement holds:
\begin{lemma}\label{l:Holder}
Let $\psi, p:X\to \R_+$ be nonnegative functions on some (at most countable) set $X$, such that for some $p_0>0$
\[
\sum_x \psi(x) \le 1, \quad \sum_x p(x) =1, \quad \forall x\in X \quad p(x)\le p_0.
\]
Then for every $\alpha<1$ one has
\begin{equation}\label{eq:p-bound}
\sum_{x\in X} \psi(x)^{\alpha} p(x) \le p_0^{\alpha}.
\end{equation}
\end{lemma}
\begin{proof}
Note that the functions $\psi(x)^{\alpha}$ and $p(x)^{1-\alpha}$ belong to the $\ell_{\frac{1}{\alpha}}(X)$ and $\ell_{\frac{1}{1-\alpha}}(X)$ respectively, and their norms do not exceed 1:
\[
\sum_x (\psi(x)^{\alpha})^{1/\alpha} = \sum_x \psi(x) \le 1, \qquad  \sum_x (p(x)^{1-\alpha})^{\frac{1}{1-\alpha}} = \sum_x p(x) =1.
\]
Applying H\"older inequality, one gets 
\[
\sum_{x\in X} \psi(x)^{\alpha} p(x)^{1-\alpha} \le \left\| \psi^{\alpha}\right\|_{\frac{1}{\alpha}} \cdot \left\| p^{1-\alpha}\right\|_{\frac{1}{1-\alpha}} \le 1\cdot 1 =1.
\]
Now, from the assumption of the lemma one has
\[
\forall x\in X \quad p(x) \le p(x)^{1-\alpha} \cdot p_0^{\alpha},
\]
and hence 
\[
\sum_{x\in X} \psi(x)^{\alpha} p(x) \le \left(\sum_{x\in X} \psi(x)^{\alpha} p(x)^{1-\alpha}\right) \cdot p_0^{\alpha} \le 1\cdot p_0^{\alpha}=p_0^{\alpha},
\]
thus concluding the proof.
\end{proof}

Let us now apply Lemma~\ref{l:Holder} to our case: here $X$ is the $G$-orbit of $I_c$, the function $\psi$ is given by $\psi(I')=|I'|$, and $p(I'):=\Prob(I_n=I')$, bounded from above by $p_0:=\frac{C_3}{n^d}$ due to the upper bound~\eqref{eq:I-upper}. We get 
\begin{equation}\label{eq:E-n}
\E |I_n|^{\alpha} = \sum_{I'} |I'|^{\alpha} \cdot \Prob(I_n=I') \le p_0^{\alpha} = \frac{C_3^{\alpha}}{n^{\alpha d}}.
\end{equation}
As $\alpha d>1$ due to the assumption of the lemma, summing~\eqref{eq:E-n} over $n$, we get 
\[
\sum_n \E |I_n|^{\alpha} \le C_3^{\alpha} \cdot \sum_n \frac{1}{n^{\alpha d}} < \infty,
\]
thus concluding the proof.
\end{proof}

\begin{proof}[Proof of Proposition~\ref{p:process}]
Fix a sequence $n_j=2^j$; we will construct the path as a concatenation of independent paths of lengths $n_1,n_2,\dots$. Namely, let 
\[
s_0:=0, \quad s_{j+1}=s_j+n_j, \quad j=1,2,\dots.
\]
Now, for each $j$, take random variables $f_{w_{j,1}},\dots,f_{w_{j,n_j}}\in \mF$ such that the product 
\[
g_{(j)}:=f_{w_{j,n}},\dots,f_{w_{j,n_1}}\in G
\]
is distributed uniformly in the ball $B_{n_j}^G(e)$ of radius~$n_j$, and such that these families $(f_{w_{j,1}},\dots,f_{w_{j,n_j}})$ are independent for different~$j$. One way of doing so is to \emph{first} choose the products $g_{(j)}\in B_{n_j}^G(e)$ independently (and uniformly in the corresponding balls). Then, one deterministically takes for each $g_j$ the lexicographically first way of decomposing it into a product of exactly $n_j$ elements from~$\mF$.

Now, we take 
\[
w_{s_j+m}:=w_{j,m}, \quad m=1,\dots, n_j;
\]
in other words, the path writes as 
\[
\dots \underbrace{f_{w_{j,n_j}} \dots f_{w_{j,1}}}_{g_{(j)}}\dots \underbrace{f_{w_{2,2}} f_{w_{2,1}}}_{g_{(2)}} \underbrace{f_{w_{1,1}} }_{g_{(1)}}.
\]

To conclude the proof, recall that the (group) convolution cannot increase $\|\cdot \|_{\infty}$-norm of a measure;
formally, we have the following (immediate) lemma:
\begin{lemma}\label{l:prod}
Let $\xi_1,\xi_2$ be two independent random variables, taking values in a group~$G$, and 
\[
\xi=\xi_1 \xi_2
\]
be their product. Then
\[
\max_{g\in G} \Prob(\xi=g) \le \max\left(\max_{g\in G} \Prob(\xi_1=g), \max_{g\in G} \Prob(\xi_2=g)\right).
\]
The same applies for the product $\xi=\xi_1\dots \xi_l$ of more than two independent random variables.
\end{lemma}

Now, for every $n=1,2,\dots$ one can write $n=s_j+m$, where $j$ is taken so that $s_j\le n < s_{j+1}$ and hence $0\le m< n_j$. The corresponding product $g_{n}$ is then equal to
\[
g_{n} =\underbrace{(f_{w_{j,m}}\dots f_{w_{j,1}})}_{\xi_1} \cdot \underbrace{g_{(j)}}_{\xi_2}\cdot \underbrace{g_{(j-1)}\cdot  \dots \cdot g_{(1)}}_{\xi_3}.
\] 
Note that by construction 
\[
\xi_1:=f_{w_{j,m}}\dots f_{w_{j,1}}, \quad \xi_2:=g_{(j)} \quad \text{and} \quad \xi_3:=g_{(j-1)}\cdot  \dots \cdot g_{(1)}
\]
are independent. Lemma~\ref{l:prod} together with the choice of $g_{(j)}$ thus implies that 
\[
\forall g\in G \quad \Prob(\xi_1\xi_2\xi_3=g) \le \frac{1}{B_{n_j}^G(e)}.
\]
Finally, $n_j = \frac{s_j+1}{2}$, and $s_j\ge \frac{n}{2}$, thus implying $n_j\ge \frac{n}{4}$. Hence,
\[
\forall g\in G \quad \Prob(g_n=g) \le \frac{1}{B_{n/4}^G(e)}.
\]
This completes the proof of Proposition~\ref{p:process}, and hence of Proposition~\ref{l:alpha-conv} and of Proposition~\ref{thm upper bound} for the case of compact~$M$.
\end{proof}

\subsection{Case of $M=(0,1]$}\label{s:half-int}

If $M$ is a half-open interval, without loss of generality, we can consider $M=(0,1]$. As $M$ is non-compact, the Generalised Kopell's Lemma cannot be applied directly: the distortion control is no longer uniform close to the open end of the interval. However, it suffices to require that the images of the initial interval do not accumulate to this end:
\begin{remark}\label{r:Kopell-half-open}
Proposition~\ref{p:Kopell} holds for non-compact $M$ under an additional assumption that all the intermediate images $f_{i_j}\cdots f_{i_1}(I_g)$ belong to some compact subset $M^\prime\subset M$.
\end{remark}
Indeed, with this additional assumption the proof can be repeated identically.

Now, to prove Proposition~\ref{thm upper bound} for this case, it suffices to modify Proposition~\ref{p:process} in a way that ensures that the images $g_n(I_c)$ stay within a compact part of $M=(0,1]$. Indeed, the construction can be modified as follows.

\begin{prop}\label{p:process-new}
Let $G$ be a finitely generated group, acting on $M=(0,1]$, and some point~$x_0$ is chosen. Also, let $\mF$ the set, formed by generators of~$G$, their inverses and~$e$. Then there exists a random process 
\[
g_0=e, \quad g_{n}=f_{w_n}g_{n-1}, \quad f_{w_n}\in\mF, \quad n=1,2,\dots,
\]
such that for every $n$ and any $g\in G$, one has 
\begin{equation}\label{eq:p-n-upper-new}
\Prob(g_n=g)\le \frac{1}{\# B^G_{n/8}(e)},
\end{equation}
as well as for every $n$ one has $g_n(x_0)\ge x_0$.
\end{prop}

Proposition~\ref{p:process-new} implies an analogue of Proposition~\ref{l:alpha-conv}:
\begin{lemma}\label{l:alpha-conv-new}
In the assumptions above, for every $\alpha>[\growth(G/K)]^{-1}$ there exists a path 
\[
g_0=e, \quad g_{n}=f_{w_n}g_{n-1}, \quad n=1,2,\dots
\]
in the group $G$, such that
\begin{equation}\label{eq:g-sum-new}
\sum_n |g_n(I_c)|^{\alpha} < \infty.
\end{equation}
and that for some compact $M^\prime\subset M$ one has $g_n(I_c)\subset M^\prime$.
\end{lemma}
\begin{proof}[Deducing Lemma~\ref{l:alpha-conv-new} from Proposition~\ref{p:process-new}]
Take $x_0$ to be the left end point of $I_c$. Then for the process $g_n$, provided by Proposition~\ref{p:process-new}, one has 
\[
g_n(I_c)\subset [g_n(x_0),1]\subset [x_0,1]=:M^\prime.
\]
The rest of the proof of Proposition~\ref{l:alpha-conv} can be repeated word-for-word.
\end{proof}

Hence, once proven, Proposition~\ref{p:process-new} implies Proposition~\ref{thm upper bound} for the case of half-open interval $M=(0,1]$, as Lemma~\ref{l:alpha-conv-new} brings us in the assumptions of the Generalised Kopell's Lemma (in its form in Remark~\ref{r:Kopell-half-open}).

\begin{proof}[Proof of Proposition~\ref{p:process-new}]
Consider the most-moving to the right sequence of iterates of the initial point $x_0$: let the sequences $(u_n,x_n)$ be defined by 
\[
x_n:=f_{u_{n}}(x_{n-1}) = \max_{f\in \mF} f(x_{n-1}), \quad n=1,2\dots.
\]
Let also 
\[
F_n:=f_{u_n}\dots f_{u_1}
\]
be the corresponding composition of length~$n$.
\begin{lemma}\label{l:F-g}
For any composition $g$ of elements of $\mF$ of length at most $n$, and any $x\ge x_0$, one has 
\begin{equation}\label{eq:right}
g(F_n (x)) \ge x_0.   
\end{equation}
\end{lemma}
\begin{proof}
Due to the monotonicity it suffices to establish~\eqref{eq:right} for $x=x_0$. Applying $g^{-1}$ to both sides, we see that it is equivalent to
\begin{equation}\label{eq:F-g}
F_n(x_0)\ge g^{-1}(x_0);
\end{equation}
as $\mF$ contains $e$ (so we can assume the length to be exactly equal to~$n$) and is preserved by passing to the inverses, we can write $g^{-1}=f_{u'_n}\dots f_{u'_1}$, thus transforming~\eqref{eq:F-g} into
\[
x_n=f_{u_n}\dots f_{u_1} (x_0) \ge f_{u'_n}\dots f_{u'_1} (x_0),
\]
This inequality can be proven by induction, using the definition of $f_{u_n}$:
\[
f_{u'_n}(f_{u'_{n-1}}\dots f_{u'_1} (x_0)) \le  f_{u'_n}(f_{u_{n-1}}\dots f_{u_1} (x_0))=f_{u'_n} (x_{n-1}) \le f_{u_n}(x_{n-1}) =x_n.
\]
\end{proof}

We are now ready to construct the desired sequence $w_n$. To do so, define random variables (independent for different $j$)
\[
g_{(j)}= f_{w_{j,s_j}}\dots f_{w_{j,1}}
\]
in the same way as in the proof of Proposition~\ref{p:process}. Now, we take
\[
w_{2s_j+m}:=u_{m}, \quad w_{s_j+n_j+m}:=w_{j,m}, \quad m=1,\dots, n_j;
\]
in other words, the path writes as 
\[
\dots \underbrace{f_{w_{j,n_j}} \dots f_{w_{j,1}}}_{g_{(j)}}\underbrace{f_{u_{n_j}} \dots f_{u_{1}}}_{F_{n_j}}\dots \underbrace{f_{w_{2,2}} f_{w_{2,1}}}_{g_{(2)}}\underbrace{f_{u_{2}}f_{u_{1}}}_{F_2} \underbrace{f_{w_{1,1}} }_{g_{(1)}} \underbrace{f_{u_{1}}}_{F_1}.
\]
Then the occurrences of $F_{n_j}$ ensure that the condition $g_n(x_0)\ge x_0$ is fulfilled: $g_{(j)}(F_{n_j}(x_0))\ge x_0$ due to Lemma~\ref{l:F-g}.

At the same time, for any length $n$, take $j$ such that $2s_j\le n<2s_{j+1}$. Then, in the same way as in the proof of Proposition~\ref{l:alpha-conv}, the product $g_{(j)}$ is an independent part of the product defining~$g_n$, thus
\[
\forall g\in G \quad \Prob(g_n=g)\le \frac{1}{\# B^G_{n_j}(e)},
\]
At the same time, we get $n_j = \frac{s_j+1}{2} \ge \frac{n}{8}$, thus concluding the proof.
\end{proof}

\subsection{Critical case}\label{s:critical} 

The idea of the proof of this proposition goes back to \cite{na-crit}. Namely, for the critical value $\alpha_0:=[\growth(G/\textup{Stab}(I_c))]^{-1}$ we no longer have the \emph{uniform} bound~\eqref{eq:g-sum} from Proposition~\ref{l:alpha-conv} to apply Generalized Kopell's Lemma~\ref{p:Kopell}. However, the upper bounds for the sum of $\alpha$'s powers of lengths during the first $n$ iterations can be shown to grow sufficiently slowly, so that when one repeats the steps of the proof of Proposition~\ref{p:Kopell}, the resulting estimate for the derivative $D(g^m)$ grows sublinearly, as~$o(m)$. On the other hand, there is a very general statement, stating that $\|D(g^m)\|=o(m)$ is impossible even for $C^1$-diffeomorphisms. This contradiction implies that the action of regularity $\alpha_0$ is also impossible.

Let us start with the lower bound for the growth of norm of derivatives. The following lemma was noticed by L.~Polterovich and M.~Sodin~\cite{ps}; for the reader's convenience we recall here its proof: 
\begin{lemma}[Polterovich-Sodin, {\cite{ps}}]\label{l:no-sublinear}
Let $J$ be a closed interval, $g\in\Diff_+^1(J)$, and assume $g\neq \id$. Then the comparison 
\[
\|Dg^m\|_{C^1(J)} = o(m)
\]
cannot take place.
\end{lemma}
\begin{proof}[Proof of Lemma~\ref{l:no-sublinear}]
Take any point $x_0\in J$ that is not fixed by $g$; without loss of generality we can assume that $g(x_0)>x_0$. Let 
\[
J'_0 := [x_0,g(x_0)]
\]
be the corresponding fundamental domain for the action of $g$, and let 
\[
J'_j:=g^{-j}(J'_0), \quad j=1,2,\dots
\]
be the corresponding preimages. Then for any $n$ the interiors of the intervals $J'_n,\dots,J'_{2n}$ are disjoint, hence 
\begin{equation}\label{eq:m-def}
\exists m, \quad n\le m \le 2n : \quad |J'_m| \le \frac{1}{n} |J|.
\end{equation}
As $g^m(J'_m)=J'_0$, by Lagrange Theorem we have 
\[
\exists x\in J'_m : \quad (Dg^m)(x) = \frac{|J'_0|}{|J'_m|},
\]
thus implying (together with the inequality $n\ge \frac{m}{2}$) the lower bound
\begin{equation}\label{eq:Dgm-lower}
\| Dg^m \|_{C^1(J)} \ge (Dg^m)(x) \ge \frac{|J'_0|}{\frac{1}{n} |J|} \ge \frac{|J'_0|}{2|J|} \cdot m.
\end{equation}
The sequence of iterations $m=m_n$, constructed by~\eqref{eq:m-def}, tends to infinity due to the inequality $m\ge n$. Hence the lower bound~\eqref{eq:Dgm-lower} takes place for an infinite number of $m$'s, thus concluding the proof.
\end{proof}

To establish the upper bound for the distortion, we replace Proposition~\ref{l:alpha-conv} by the following one: 
\begin{prop}\label{p:process-critical}
In the assumptions of Proposition~\ref{l:alpha-conv} for $d:=[\growth(G/K)]$ there exist constants $C_1, C_2$, such that for every~$N$ 
there exists a path 
\[
g_0=e, \quad g_{n}=f_{w_n}g_{n-1}, \quad n=1,2,\dots, N,
\]
in the group $G$, for which
\begin{equation}\label{eq:g-sum-N}
\sum_{n=0}^{N} |g_n(I_c)|^{1/d}  \le C_1 \log^{1-\frac{1}{d}} N
\end{equation}
and 
\begin{equation}\label{eq:length-N}
|g_N(I_c)| \le \frac{C_2}{N^d}.
\end{equation}
The same holds in the assumptions of Proposition~\ref{p:process-new}, with the additional conclusion that all $g_n(I_c)$ are to the right of~$I_c$.
\end{prop}

Let us deduce Proposition~\ref{p:critical} from these estimates (postponing the proof of Proposition~\ref{p:process-critical} to the end of this section):

\begin{proof}[Proof of Proposition~\ref{p:critical}]
Take and fix $g=\phi(c)$, and let $I_c$ be its support interval. 
For any $m\in \N$ and any $x\in I_c$, let us estimate $Dg^m(x)$; the estimate that we will obtain would be of the form~$o(m)$, thus providing a contradiction with the conclusion of Lemma~\ref{l:no-sublinear}.

Let~$y=g^m(x)$; then, for any map $F\in \phi(G)$, we have
\begin{equation}\label{eq:Dgm-critical}
Dg^m (x) = Dg^m |_{F_n(x)} \cdot \frac{DF_n (x)}{DF_n(y)}.
\end{equation}
Take $F$ to be the composition provided by Proposition~\ref{p:process-critical} for $N=m$. Then, the inequality~\eqref{eq:g-sum-N}, together with the distortion control~\eqref{eq:kappa-subadditive},, implies an upper estimate for the second factor in the right hand side of~\eqref{eq:Dgm-critical}:  
\begin{equation}\label{eq:DF-quotient}
\frac{DF_n (x)}{DF_n(y)} \le \exp \left(C_{\kappa} \sum_{j=0}^{m-1} |g_n(I_c)|^{1/d}\right) \le  \exp \left(C_{\kappa} \cdot C_1 \log^{1-\frac{1}{d}} m \right)  = o(m).
\end{equation}

On the other hand, the estimate~\eqref{eq:length-N} for the image length $|g_m(I_c)|$, together with the fact that~$g^m$ preserves $g_m(I_c)$ and hence its log-derivative vanishes at some point, imply (due to the same distortion arguments) that
\begin{equation}\label{eq:Dgm-image}
\|Dg^m\|_{C(g_m(I_c))} \le \exp \left(C_{\kappa} \cdot m\cdot |g_m(I_c)|^{1/d}\right) \le \exp\left(C_{\kappa} C_2^{1/d}\right) = O(1).
\end{equation}
Joining the two estimates~\eqref{eq:DF-quotient} and~\eqref{eq:Dgm-image}, we get from~\eqref{eq:Dgm-critical} that $\|Dg^m\|_{C(I_c)} = o(m)$, obtaining the desired contradiction.
\end{proof}

\begin{proof}[Proof of Proposition~\ref{p:process-critical}]
Take $N$ steps of the random process provided by Proposition~\ref{p:process}, and let $\xi_1$, $\xi_2$ be the random variables giving the values in the left hand sides of~\eqref{eq:g-sum-N} and~\eqref{eq:length-N}:
\[
\xi_1:= \sum_{n=0}^{N} |g_n(I_c)|^{1/d}, \quad \xi_2:=|g_N(I_c)|.
\]
Then it will suffice to show that the \emph{expectations} of these random variables admit the same kind of bounds. Indeed, by Markov inequality, for a positive random variable $\xi$ one has $\Prob (\omega : \, \xi(\omega)>3 \E\xi)\le \frac{1}{3}$. Hence, if 
\begin{equation}\label{eq:xi-expect}
\E \xi_1 \le C'_1 \log^{1-\frac{1}{d}} N, \quad \E \xi_2 \le \frac{C'_2}{N^d},
\end{equation}
then for at least some $\omega$ (actually, for the set of probability at least $1-\frac{1}{3}-\frac{1}{3}=\frac{1}{3}$) one has 
\[
\xi_1 \le 3C'_1 \log^{1-\frac{1}{d}} N, \quad \xi_2 \le \frac{3C'_2}{N^d},
\]
thus establishing the desired~\eqref{eq:g-sum-N} and~\eqref{eq:length-N} with $C_1=3C'_1, \, C_2=3C'_2$.

Now, the second part of~\eqref{eq:xi-expect} is easy to obtain. Indeed, identifying $G/K=G/\Stab(I_c)$ with the $G$-orbit of $I_c$,
\[
\Orb_{I_c}:= \{ f(I_c) \, : f\in G\},
\]
we get 
\[
\E \xi_2 = \sum_{J \in \Orb_{I_c}} |I| \cdot \Prob(g_N(I_c)=J).
\]
The right hand side is a scalar product of the ``length'' vector $L:=(|J|)_{J\in \Orb_{I_c}}  \in \ell_1(\Orb_{I_c})$ and the probability vector 
$p_N:=(\Prob(g_N(I_c)=J))_{J\in \Orb_{I_c}} \in \ell_{\infty}(\Orb_{I_c})$, we get 

\[
\E \xi_2 = \langle L, p_N \rangle \le \|L\|_{1} \cdot \|p_N\|_{\infty} \le |I| \cdot \frac{1}{\# B^{G/K}_{N/4}(e)} \le \frac{C'_1}{N^d},
\]
where the last estimate is due to the polynomial growth of~$G/K$. This completes the proof the second part of~\eqref{eq:xi-expect} (leading to~\eqref{eq:length-N} via the Markov inequality).

In order to estimate the expectation of~$\xi_1$, let us again write it in terms of a scalar product: 
\begin{equation}\label{eq:expect-xi1}
\E \xi_1 = \sum_{n=0}^N \E |g_n(I_c)|^{1/d} = \langle \wL, \wtp_N \rangle,
\end{equation}
where 
\begin{equation}\label{eq:wL-wp}
\wL:= (|J|^{1/d})_{J\in \Orb_{I_c}} \in \ell_{d}(\Orb_{I_c}), \quad  \wtp_N := \left( \sum_{n=0}^N \Prob(g_n(I_c)=J) \right)_{J\in \Orb_{I_c}}.
\end{equation}
We are now going to apply H\"older inequality to estimate the product in the right hand side of~\eqref{eq:expect-xi1},
\[
\langle \wL, \wtp_N \rangle \le \|\wL\|_{\ell_{d}(\Orb_{I_c})} \cdot  \| \wtp_N\|_{\ell_{q}(\Orb_{I_c})} \le  |I|^{1/d} \cdot  \| \wtp_N\|_{\ell_{q}(\Orb_{I_c})},
\]
where $q:=\frac{d}{d-1}$ so that $\frac{1}{d}+\frac{1}{q}=1$; to do so, we will establish an upper bound for the norm~$\| \wtp_N\|_{\ell_{q}(\Orb_{I_c})}$. 

To do so, let us first estimate the values~$\wtp_N(J)$ at a given interval of the orbit $J\in \Orb_{I_c}$. 

Namely, let $n_0:=\dist(I_c,J)$; then the summands in the definition~\eqref{eq:wL-wp} of $\wtp_N(J)$ with $n<n_0$ vanish, hence
\[
\wtp_N(J) = \sum_{n=n_0}^N \Prob(g_n(I_c)=J) \le \sum_{n=n_0}^N \frac{\const}{n^d} \le \frac{C_2'}{n_0^{d-1}}.
\]
Now, 
\begin{multline}\label{eq:sum-q}
\sum_{J\in \Orb_{I_c}} \wtp_N(J)^{q} \le \sum_{r=0}^{[\log_2 N]} \sum_{2^{r} \le \dist(J,I_c) < 2^{r+1}} \frac{(C_2')^q}{\dist(J,I_c)^d} 
\\
\le \sum_{r=0}^{[\log_2 N]} \frac{(C_2')^q}{(2^r)^d} \cdot \# B^{G/K}_{2^{r+1}}(e) \le \const  \sum_{r=0}^{[\log_2 N]} \frac{(2\cdot 2^r)^d}{(2^r)^d} \le C_3 \log N
\end{multline}
for some constant~$C_3$. The bound~\eqref{eq:sum-q} implies 
\[
\|\wtp_N\|_{\ell_{q}(\Orb_{I_c})} \le \const \cdot (\log N)^{1/q} = \const \cdot \log^{1-\frac{1}{d}} N,
\]
thus completing the estimate of $\E\xi_1$ in~\eqref{eq:xi-expect}. As~\eqref{eq:xi-expect} is now fully established, an application of the Markov inequality completes the proof of the proposition.

Using the conclusions of Lemma~\ref{l:alpha-conv-new} instead of Proposition~\ref{l:alpha-conv}, by repeating word-for-word the above arguments, one concludes the proof for the case of the action on~$M=(0,1]$. 
\end{proof}

\section{Lower bounds: construction of the actions}\label{s:lower}

\subsection{Pixton--Tsuboi strategy}\label{s:P-T}

In order to construct actions of regularity $C^{1+\alpha}$, we are going to use the Pixton--Tsuboi strategy (see \cite{pixton} and \cite{tsuboi}). For the reader's convenience, we recall here its ideas. Start by assuming that we are given an (orientation-preserving) action of the group $G$ on a countable ordered set~$(X,\prec)$ (``the set of intervals''), where each point has a closest neighbour from above and from below. The latter condition can be rephrased by saying that $X$ is of the form~$X=X_0\times \Z$, where $X_0$ is a fully ordered set, and the order on $X$ is lexicographic: $(v,j)\prec (v',j')$ if and only if $v\prec v'$ or $v=v'$ and $j<j'$. Note that in this case the action of $G$ on $X$ descends to an action on $X_0$, and the original action is of the form
\begin{equation}\label{eq:l-g}
g(v,j)=(g(v), j+ \lh(g,v)),
\end{equation}
where $\lh(\cdot,\cdot)$ satisfies the cocycle relation 
\[
\lh(g'g,v) = \lh(g',g(v)) + \lh (g,v).
\]

Choose a strictly positive function $L\in \ell_1(X)$, and replace each point $(v,j)\in X$ by an interval~$I_{v,j}$ of length $L(v,j)$, so that these intervals are adjacent to each other: the left endpoint of $I_{v,j}$ is the point $x_{v,j}:=\sum_{(u,k)\prec (v,j)} L(u,k)$. Then, the group is naturally acting on the closure of the set of the boundary points $\bd:=\overline{\{x_{(v,j)} \mid (v,j)\in X\}}$. 

Next step is to extend this action inside the intervals, in a way that the resulting action on the interval 
\[
I_M:=[\inf \bd, \sup \bd] = [0, \sum_{u,k} L(u,k)]
\] 
would be sufficiently regular; see Fig.~\ref{f:intervals}.

\begin{figure}[!h!]
{\center
\includegraphics[width=15cm]{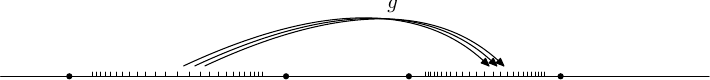} 

}
\caption{A family of intervals, permuted by the action}\label{f:intervals}
\end{figure}

\subsection{Tsuboi family}\label{s:T-family}

The extension inside~$I_M$ is made with help of the \emph{Tsuboi family} of interval diffeomorphisms. Namely, there exists the following family of maps.
\begin{prop}[Tsuboi, \cite{tsuboi}]\label{p:Tsuboi}
There is a way to associate to any four intervals $I',I,J',J$, where $I'$ is adjacent to the left to $I$, and $J'$ is adjacent to the left to $J$, a diffeomorphism $\varphi_{I',I}^{J',J}:I\to J$, such that the following conditions hold: 
\begin{enumerate}
\item\label{i:derivative}  the derivatives of $\varphi_{I',I}^{J',J}$ at the left and right endpoints of $I$ are respectively equal to~$\frac{|J'|}{|I'|}$ and to~$\frac{|J|}{|I|}$ (see Fig.~\ref{f:Tsuboi}, left).
\item\label{i:composition} for any three pairs of adjacent intervals $I',I;J',J;K',K$ one has 
\begin{equation}\label{eq:comp}
\varphi_{J',J}^{K',K}\circ \varphi_{I',I}^{J',J}=\varphi_{I',I}^{K',K}.
\end{equation}
\item\label{i:bound-M} For every constant $R>1$ there is a constant $C_R$ such that if 
\begin{equation}\label{eq:1/2}
\frac{1}{R}\le \frac{|J|}{|I|} \le R, \quad \frac{1}{R}\le \frac{|J'|}{|I'|} \le R, 
\end{equation}
then one has 
\begin{equation}\label{eq:M}
\forall x,y\in I \quad \left|\log D \varphi_{I',I}^{J',J}(x)-\log D \varphi_{I',I}^{J',J}(y)\right| \le C_R \cdot \left| \log \frac{|J|}{|I|} - \log \frac{|J'|}{|I'|} \right| \cdot \frac{|x-y|}{|I|}.
\end{equation}
\end{enumerate}
\end{prop}

\begin{figure}[!h!]
{\center
\includegraphics[height=5cm]{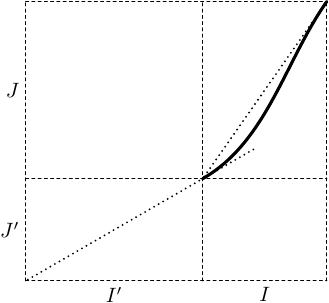} \qquad \includegraphics[width=8cm]{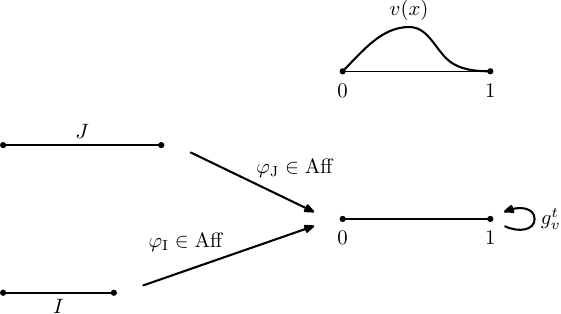}

}
\caption{An map $\varphi_{I',I}^{J',J}$ from the Tsuboi family: its graph over $I$ and tangent lines at the endpoints (left), and its construction (right)}\label{f:Tsuboi}
\end{figure}

\begin{proof}[Sketch of the proof]
For the reader's convenience, we recall here a sketch of Tsuboi's construction (the details can be found in \cite{tsuboi}). Namely, one fixes on the interval $[0,1]$ a smooth vector field $V(x)$, vanishing at the endpoints, such that $x=0$ is its hyperbolic fixed point with the Lyapunov exponent equal to~$1$, while $x=1$ is parabolic; for instance, one can take $V(x)=x(1-x)^2 \frac{\partial}{\partial x}$. Then, every $\varphi_{I',I}^{J',J}$ is constructed as follows. For every interval $I_0$, let $\varphi_{I_0}:I_0\to [0,1]$ be the (uniquely defined) orientation-preserving affine bijection. Then, one lets 
\[
\varphi_{I',I}^{J',J}:= \varphi_{J}^{-1} \circ \exp(tV) \circ \varphi_I,
\]
where $\exp(tV)$ is the time $t$ shift along the vector field $V$, and the time $t$ is 
chosen in a way that ensures the condition for the derivative at the left endpoint (see Fig.~\ref{f:Tsuboi}, right). It is easy to see that the time $t$ is uniquely defined by such a condition, namely, one has 
\[
t=t(I',I,J',J) = \log \left(\frac{|J'|/|I'|}{|J|/|I|} \right) = \log \frac{|J'|}{|I'|} - \log \frac{|J|}{|I|}.
\]
Then, conclusion~\ref{i:derivative} is guaranteed by the construction, and the conclusion~\ref{i:composition} follows easily from it (the composition in the right hand side of~\eqref{eq:comp} is again of the same type, as the affine maps $\varphi_J$ and $\varphi_J^{-1}$ in the middle cancel out). 

The assumption~\eqref{eq:1/2} then provides a uniform upper bound for the time~$t$. Now, this upper bound implies an upper bound for the Lipschitz constant of $\log D(\exp(tV))$ on $[0,1]$; on the other hand, as $\varphi_I,\varphi_J$ are affine, one has 
\[
\left.\left(\log D \varphi_{I',I}^{J',J}\right)\right|_x^y=  \left.\left(\log D\exp(tV) \right)\right|_{\varphi_I(x)}^{\varphi_I(y)} \quad \text{and} \quad \left| \varphi_I(x)-\varphi_I(y) \right| = \frac{|x-y|}{|I|}.
\]
In turn, this implies the estimate~\eqref{eq:M}.
\end{proof}

\begin{remark}\label{r:C-R}
As the map $\varphi=\varphi_{I',I}^{J',J}$ has derivatives $\frac{|J|}{|I|}$ and $\frac{|J'|}{|I'|}$ at the endpoints of~$I$, for any $\alpha\le 1$ the $\alpha$-H\"older constant of its log-derivative $\log D\varphi$ cannot be smaller than $\frac{1}{|I|^{\alpha}} \left| \log \frac{|J|}{|I|} - \log \frac{|J'|}{|I'|} \right|$. It easily follows from~\eqref{eq:M} that (in the assumptions of conclusion~\ref{i:bound-M}) this estimate is exact up to a constant factor. Namely, for the same constant $C_R$ as in~\eqref{eq:1/2}, one has
\begin{equation}\label{eq:M-alpha}
\forall x,y\in I\quad \left|\log D \varphi_{I',I}^{J',J}(x)-\log D \varphi_{I',I}^{J',J}(y)\right| \le C_R \cdot \frac{\left| \log \frac{|J|}{|I|} - \log \frac{|J'|}{|I'|} \right|}{|I|^{\alpha}} \cdot |x-y|^{\alpha}.
\end{equation}
\end{remark}

Given this family of maps, one can extend the above action of the group $G$ from the set $\bd$ to all the intervals $I_{v,j}$. Namely, one chooses
\begin{equation}\label{eq:g-I-v-j}
g|_{I_{v,j}} = \varphi_{I_{v,j-1},I_{v,j}}^{I_{u,k-1},I_{u,k}},
\end{equation}
where $g(v,j)=(u,k)$ in the sense of the action of $G$ on $X=X_0\times \Z$. Note that every such restriction $g|_{I_{v,j}}$ is a $C^{\infty}$-diffeomorphism on the image, and the derivatives at the left and at the right at each endpoint $x_{v,j}$ (separating $I_{v,j-1}$ and $I_{v,j}$) coincide due to the condition~\ref{i:derivative} of the Tsuboi family: both are equal to $\frac{|I_{u,k-1}|}{|I_{v,j-1}|}$. In particular, the constructed maps are $C^{1}$-diffeomorphisms of \emph{open} intervals 
\begin{equation}\label{eq:I-v}
I_v:=\bigcup_{j\in\Z} I_{v,j}. 
\end{equation}
The next part is to choose the lengths $L(v,j)=|I_{v,j}|$ in such a way that these maps would be of class $C^{1+\alpha}$.

\subsection{Choice of lengths and estimates for the H\"older constants}\label{s:Choice-and-control}

\subsubsection{Choice of lengths and the structure of $I_v$}

A choice of lengths, related to the previous works (see for instance \cite{int4}, \cite{ER}, \cite{JNR}, \cite{kkr}), is the following one. For every $\eps\in (0,1)$ and $A\ge 1$, define the sequence of lengths 
\begin{equation}\label{eq:L-A-eps}
L_{A,\eps}(j):= \frac{1}{(A^2 +j^2)^{\frac{1+\eps}{2}}}.
\end{equation}
Now, assume that $\eps\in (0,1)$ is fixed, and one chooses $A_v\ge 1$ for every $v\in X_0$ (in other words, the map $A_{\bullet}:X_0\to [1,+\infty)$). Then, let the interval lengths be defined by
\begin{equation}\label{eq:L-choice-A-v}
L(v,j):=L_{A_v,\eps}(j).
\end{equation}

The choice~\eqref{eq:L-A-eps} is a slightly optimized version of the choices in the cited papers. We could have also used $\frac{1}{A^{1+\eps} + j^{1+\eps}}$, but it would lead to slightly worse estimates.

Two immediate remarks on the choice~\eqref{eq:L-A-eps} is that the intervals $I_{v,j}\subset I_v$ with $|j|\le 3A_v$ have lengths comparable to $A_v^{-(1+\eps)}$, while those with $|j|\ge A_v$ have lengths comparable to $j^{-(1+\eps)}$: 
\begin{equation}\label{eq:length-core}
\frac{1}{6A_v^{1+\eps}} \le L_{A,\eps}(j) \le \frac{1}{A_v^{1+\eps}}, \quad |j|\le 3A_v,
\end{equation}
\begin{equation}\label{eq:length-outer}
\frac{1}{6j_v^{1+\eps}} \le L_{A,\eps}(j) \le \frac{1}{j^{1+\eps}}, \quad |j|> A_v.
\end{equation}
In further, we will refer to the intervals $I_{v,j}\subset I_v$ with $|j|\le A_v$ as the \emph{core part}, and with $|j|> A_v$ as \emph{the outer part} (see Fig.~\ref{fig:lvj}).
Finally, note that the intervals $I_{v,j}$ and $I_{v,j+1}$ have lengths with the quotient close to~$1$; we will be using it in the estimates below.

\begin{figure}[h!]
\begin{center}
\includegraphics{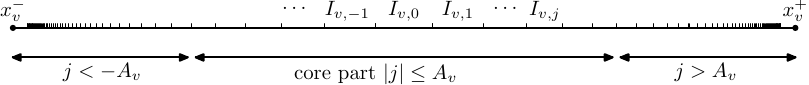}
\end{center}
\caption{The interval $I_v$ for given $A_v$.}\label{fig:lvj}
\end{figure}

\subsubsection{Reducing to study of individual intervals $I_v$}

A first remark is that for any two such intervals $I_v$, $I_{u}$ and any $l$, the map $g:I_v\to I_u$, defined by its restrictions on every $I_{v,j}$ by~\eqref{eq:g-I-v-j}, has its derivative converging to~$1$ at the endpoints of $I_v$.
\begin{lemma}
Assume that the lengths $L(v,j)$ are given by~\eqref{eq:L-A-eps} and~\eqref{eq:L-choice-A-v}, and for some $u=g(v)$ and $l_0=l(v,g)$ the action of $g$ on $I_v$ is defined by~\eqref{eq:g-I-v-j} and~\eqref{eq:l-g}. Then the restriction $g:I_v\to I_u$ is a $C^1$-diffeomorphism, and one has 
\begin{equation}\label{eq:D-limit}
\lim_{x\to \xvl} D g(x) = \lim_{x\to \xvr} D g(x) = 1,
\end{equation}
where $I_v=[\xvl,\xvr]$.
\end{lemma}
\begin{proof}
One first checks that the the derivatives~\eqref{eq:D-limit} tend to~$1$ at the points $x=x_{v,j}$. Indeed, $Dg(x_{v,j}) = \frac{L_{u,j+l}}{L_{v,j}}$, and 
\begin{equation}
\lim_{j\to \pm \infty} \frac{L_{g(v),j+\lh}}{L_{v,j}} =\lim_{j\to \pm \infty} \left(\frac{A_{v}^{2}+ l^{2}}{A_{g(v)}^{2}+(j+\lh)^{2}}\right)^{\frac{1+\eps}{2}} =1.
\end{equation}
Then, the estimate~\eqref{eq:M} allows to extend closeness of the derivatives to~$1$ inside each interval $I_{v,j}=[x_{v,j},x_{v,j+1}]$.
\end{proof}
\begin{corollary}\label{c:split}
The action of $g\in G$, defined using the lengths~\eqref{eq:L-A-eps}, is of class $C^{1+\alpha}$ if and only if it is of class $C^{1+\alpha}$ on each interval $I_v$, and the corresponding $\alpha$-H\"older constants of $\log Dg|_{I_v}$ are uniformly bounded.
\end{corollary}
\begin{proof}
``Only if'' is immediate, let us prove the ``if'' part. Note that $g$ is then a $C^1$-diffeomorphism everywhere. Indeed, extend $Dg$ from the interiors of intervals $I_{v,j}$ by~$1$ everywhere outside. This extension defines a function~$\psi:I_M\to \R$ that is \emph{continuous} due to the uniform bound on the H\"older constant (as on small intervals $Dg$ is thus uniformly close to~$1$). The union of interiors of $I_{v,j}$ is of full measure inside~$I_M$, hence $g$ is a continuous function such that $Dg$ almost everywhere coincides with a continuous function~$\psi$. Hence,~$g$ is of class~$C^1$, and $Dg=\psi$ everywhere in~$I_M$.

Finally, given two points $x<y$ that belong to two different (and thus having disjoint interiors) intervals $I_v=[\xvl,\xvr]$, $I_u=[x_u^-,x_u^+]$, one has
\[
\left. \left(\log Dg \right) \right|_x^y  = \left. \left(\log Dg \right) \right|_x^{\xvr}  + \left. \left(\log Dg \right) \right|_{\xvr}^{x_u^-}  + \left. \left(\log Dg \right) \right|_{x_y^-}^y.
\]
Assuming that $\alpha$-H\"older constants of $\log Dg$ on each $I_v$ do not exceed a certain $C'_{\alpha}$, and using $Dg(\xvr)=Dg(x_u^-)=1$, one then has 
\begin{multline}\label{eq:splitting}
\left| \log Dg (x) - \log Dg (y) \right| \le \left| \log Dg (x) - \log Dg (\xvl) \right| + \left| \log Dg (x_u^-) - \log Dg (y) \right|
\\ \le C'_{\alpha} \, |x-\xvr|^{\alpha} + C'_{\alpha} \, |x_{u}^- -y|^{\alpha} \le 2 C'_{\alpha} \, |x-y|^{\alpha},
\end{multline}
thus showing that $\log Dg$ is $\alpha$-H\"older with the H\"older constant at most~$2 C'_{\alpha}$.
\end{proof}

\subsubsection{No-shift $l=0$ case}

Due to Corollary~\ref{c:split}, an application of this technique with these choices of lengths boils down to ensuring the uniform $\alpha$-H\"older of the restrictions $g:I_v\to I_{g(v)}$. Let us see when it can be stated. We will start with the case of $l=0$, so that $g(v,j)=(u,j)$.

Prior to writing \emph{upper} bounds for the H\"older constants, let us observe how these constants can be bounded \emph{from below}. To do so, note that due to the Lagrange theorem, the map $g:I_v\to I_{u}$ has an intermediate point $\xi\in I_{v}$ such that 
\begin{equation}\label{eq:xi-A-Ap}
Dg(\xi)= \frac{|I_{u}|}{|I_{v}|}.
\end{equation}
As $Dg=1$ at the endpoints of $I_v$, there are two points in $I_v$, between which the increment of $\log Dg$ is equal to $\log \frac{|I_{u}|}{|I_{v}|}$. Hence the $\alpha$-H\"older constant of $\log Dg$ cannot be smaller than 
\begin{equation}\label{eq:alpha-lower}
\frac{\left| \log |I_v| - \log |I_u| \right|}{|I_v|^{\alpha}} = \frac{\left| \log L(A_v) - \log L(A_u) \right|}{L(A_v)^{\alpha}},
\end{equation}
where we introduce the function 
\[
L(A):= \sum_{j=-\infty}^{\infty} L_{(A,\eps)}(j) = \sum_{j=-\infty}^{\infty} \frac{1}{(A^{2}+j^{2})^{\frac{1+\eps}{2}}}
\]
that defines the lengths of the interval $I_v$ constructed using the corresponding~$A_v$.

The following proposition states that the lower bound~\eqref{eq:alpha-lower} for the $\alpha$-regularity is exact up to a constant factor.

\begin{prop}\label{p:L-A-Ap} 
For every $\eps\in (0,1)$ there exists $C=C(\eps)$ such that the following holds. Let $g:I_v\to I_u$ be given by~\eqref{eq:g-I-v-j} and~\eqref{eq:l-g} with $l=0$, where we make the choice of lengths~\eqref{eq:L-A-eps}--\eqref{eq:L-choice-A-v}). Assume also that the lengths $L(A_v)$ and~$L(A_u)$ differ at most two times. Then the $\alpha$-H\"older constant of $g:I_v\to I_u$ satisfies an upper bound
\[
\kappa_{\alpha}(g,I_v) =\| \log Dg \|_{C^{\alpha}(I_v)} \le  C_{\eps}\cdot \frac{\left| \log \frac{L(A_v)}{L(A_u)} \right|}{L(A_v)^{\alpha} }
\]
\end{prop}

We postpone its proof until the end of this section: together with the other technical estimates, it will be presented in Section~\ref{s:technical}.

\subsubsection{Same-interval $u=v$ case}

Now, a general map $g$ with $g(v,j)=(u,j+l)$ can be represented as $g=g_0 \circ \sigma_v^l$, where $\sigma_v : I_v\to I_v$ and $g_0:I_v\to I_u$ are maps that are defined by their actions
\[
\sigma_v(v,j)=(v,j+1), \quad g_0(v,j)=(u,j)
\]
on the set of endpoints~$X\cap I_v$. As Proposition~\ref{p:L-A-Ap} can be applied to estimate the H\"older constant of $\log Dg_0$, so the next natural estimate to obtain is the one for $\log D\sigma_v^l$.

Again, we start by proposing a lower bound. Note that the interval $I_v$ can be split into its ``core part'', formed by $I_{v,j}$ with $|j|\le A_v$, and two ``ourter parts'', where $j>A_v$ and $j<-A_v$ respectively (see Fig.~\ref{fig:lvj}). As we will see, these parts have lengths, comparable with the length of~$I_v$. Indeed, for the core part one has 
\begin{equation}\label{eq:A-core}
\sum_{-[A_v]}^{[A_v]} L_{A_v,\eps}(j) = \sum_{-[A_v]}^{[A_v]} \frac{1}{(A_v^{2}+j^{2})^{\frac{1+\eps}{2}}} 
\asymp (2 A_v+1) \cdot \frac{1}{A_v^{1+\eps}} \asymp \frac{1}{A_v^{\eps}};
\end{equation}
recall that $\asymp$ means that both sides differ from each other less than a constant number of times (with a constant that does not depend on $A\ge  1$, though it might depend on~$\eps$ or~$\alpha$). In the same way, for the outer parts, one has 
\begin{equation}\label{eq:A-outer}
\sum_{j>A_v}   \frac{1}{(A_v^{2}+j^{2})^{\frac{1+\eps}{2}}} 
\asymp \int_{A_v}^{\infty} \frac{dj}{j^{1+\eps}} \asymp \frac{1}{A_v^{\eps}}.
\end{equation}
In particular, adding~\eqref{eq:A-core} and \eqref{eq:A-outer}, we get $L(A)\asymp A_v^{-\eps}$ as well. 

Also, the lengths of adjacent intervals $I_{v,j}$ and $I_{v,j+1}$, are quite close to each other; moreover, the lengths of any two intervals $I_{v,j}$ from the core part differ at most two times. Let us now estimate the H\"older contant for $\log D \sigma_v^l$ from below. We will assume $l>0$ (the case $l<0$ is analogous), and $l<A_v$.

Under these assumptions, the endpoint $x_{v,0}$ is shifted by at least 
\[
|x_{v,l}-x_{v,0}| \ge \frac{| l |}{2A_v^{1+\eps}} \asymp \frac{|l|}{A_v} \cdot L(A_v).
\]
The length of the interval $[\xvl,x_{v,0}]$ gets changed by an application of $\sigma_v^l$ at least by a factor 
\[
\frac{|x_{v,l}-\xvl|}{|x_{v,0}-\xvl|}\ge 1 + \frac{|x_{v,l}-x_{v,0}|}{L(A_v)},
\]
thus implying a lower bound for $\log D\sigma_v^l$ at some intermediate point $\xi\in I_v$ by 
\[
\log D\sigma_v^l (\xi) \ge \const \, \frac{l}{A_v}.
\]
Hence (as the derivatives at the endpoints are equal to~$1$), the $\alpha$-H\"older constant of $\log D\sigma_v^l$ cannot be smaller than 
\[
\frac{\const \cdot l/A_v}{L(A_v)^{\alpha}}
\]

As before, the following proposition states that this lower bound is, up to a constant factor, exact. Namely, we have the following statement.

\begin{prop}\label{p:sigma-v}
Let $\alpha\in (0,1)$. Then the map $\sigma_v$ is a $C^{1+\alpha}$-diffeomorphism of~$I_v$. Moreover, there exists constants $C_\sigma=C_\sigma(\eps)$, $C'=C'(\eps)$ such that for every $l$, $|l|<A_v$, the following statements hold:
\[
\| \log D\sigma_v^l  \|_{C(I_v)} \le C_{\sigma}
\]
and the $\alpha$-H\"older constant of $\log D\sigma_v^l$ admits an upper bound
\begin{equation}\label{eq:sigma-v-l}
\kappa_{\alpha}(\sigma_v^l,I_v) = \|\log D\sigma_v^l\|_{C^{\alpha}(I_v)}\le C' \cdot \frac{l/A_v}{ L(A_v)^{\alpha}}.
\end{equation}
\end{prop}

We conclude this section by the following lemma, that describes more precisely the function~$L(A)$, and its corollary:
\begin{lemma}\label{l:A-to-L}
For any $\eps>0$, we have 
\[
L(A)\asymp A^{-\eps}.
\]
Moreover, in the region $1\le A<A'<2A$, one has 
\[
|\log L(A) - \log L(A')| \asymp |\log A - \log A'|.
\]
\end{lemma}
\begin{corollary}
There exists $k_0$ such that for any interval $[A,A']\subset [1,+\infty)$ such that $L(A)$ and $L(A')$ differ at most twice, the interval $[A,A']$ can be split into at most $k_0$ intervals $[A_i,A_{i+1}]$ whose endpoints differ at most twice. In particular, modulo having larger constants, it suffices to prove Proposition~\ref{p:L-A-Ap} under an additional assumption that $\frac{A'}{2}\le A \le A'$.
\end{corollary}

We prove both these statements, as well as Proposition~\ref{p:L-A-Ap}, in Section~\ref{s:technical} below.

\subsection{Nilpotent group actions: concluding the construction}\label{s:concluding-realisation}

This section is devoted to the proof of the realisation statement, Proposition~\ref{p:lower-K}.

\subsubsection{Schreier graph and the function $\lh(g,v)$}

In order to prove Proposition~\ref{p:lower-K}, consider the action of $G$ on an interval $I$ provided by the proof of Proposition~\ref{p:stab-subgr}, with a support interval $J$ of $c$, such that $\Stab(J)=K$. Let $J_0$ be a fundamental domain of the action of $c$ on $J$. We will prove that actually, this action is  topologically conjugate to an action of the class $C^{1+\alpha}$ for any $\alpha<1/\growth(G/K)$.

Indeed, note that for the action constructed in this proof, the interval $I$ minus countably many points is covered by the images of the fundamental domain $\{g(J_0)\mid g\in G\}$, that are adjacent to each other, and are permuted by the action. Take $X_0$ as the coset space $G/K=G/\Stab(J)$, that is naturally identified with the ordred set of the image intervals $\{g(J) \mid g\in G\}$.   We are thus in the setting of applying Pixton--Tsuboi construction. 

To implement it, fix a system $\mF$ of generators of $G$, and for each coset $v\in X_0$ (corresponding to some interval $I_v=g(J)$) we fix one of the elements $h_v$ of the smallest length in $G$ that belongs to this coset; then, we  define $I_{v,j}:=h_v(c^j(J_0))$. The set of (pairwise disjoint) intervals $\{g(J_0) \mid g\in G\}$ is then identified with $X:=X_0\times \Z$, with order on the latter being a lexicographic. Also, as we have fixed the set of generators of $G$, the coset space $X_0=G/K$ becomes the set of vertices of the corresponding Schreier graph, and becomes equipped with the corresponding distance. In particular, we define $\|v\|:=\dist (v, [e])$.

Next, note that a function $\lh:G\times X_0\to \Z$, given by $g(I_{v,j})=I_{g(v),j+\lh(g,v)}$ is then defined. An important remark is that this function is of a polynomial growth in~$\|v\|$.

\begin{lemma}\label{l:ell}
There exists a polynomial $p$ such that for any $g\in \mF$ and any $v\in X_0$ one has $|\lh (g,v)| \le p(\|v\|)$. 
\end{lemma}
\begin{proof}
Let $g\in \mF$. Note that $\lh(g,v)$ is defined by the condition that the composition $g\circ h_v$ sends $J_0$ to $h_{g(v)} (c^{\lh(g,v)}(J_0))$; equivalently, by the condition that 
\[
g':= h_{g(v)}^{-1} \circ g \circ h_v \in \Stab(J)
\]
acts on $J$ as~$c^{\lh(g,v)}$. Now, due to its representation, $g'$ is an element of length at most 
\[
\|g'\|_{\mF} \le \|h_{v}\|_{\mF} + \|g\|_{\mF} + \|h_{g(v)}\|_{\mF}  \le 2(\|v\|+\|g\|_{\mF}) = 2( \|v\|+1).
\]
Now, recall that $H<K$ acts on $J$ trivially by construction; it thus suffices to note that for any element $g'\in K$ of a given length, the power $\lh$ of $c$ in its representation $g'=g'' c^{\lh}, \quad g''\in H$, admits a polynomial upper bound.

A theorem by Osin~\cite[Theorem~2.2]{osin} states that subgroups of finitely generated nilpotent groups are (in the algebraic sense) \emph{polynomially distorted}: fixing a set of generators $\mF_K$ for a subgroup $K<G$, one has for some polynomial $P$ that $\|g\|_{\mF_K} \le P(\|g\|_\mF)$ for every $g\in K$. (Actually, it also can be seen from the arguments from Section~\ref{s:upper-inclusion}.)

Fix the system of generators $\mF_K$ of $K$ to be $c$, joined with the generators of $H$. Then, we have 
\[
P(\|g'\|_{\mF}) \ge \|g'\|_{\mF_K} = \| g'' c^{\lh(g,v)}\|_{\mF_K} \ge |\lh(g,v)|,
\]
and thus 
\[
|\lh(g,v)| \le P( 2(\|v\|+1)) =: p(\|v\|).
\]
\end{proof}

We are now in the setting to apply the Pixton--Tsuboi construction. To do it, it suffices to choose a value $\eps>0$ and a function $A_{\bullet}:X_0\to [1,\infty)$ that defines the values $A_v$ and thus all the lengths $L(v,j)=L_{A_v,\eps}(j)$, given by~\eqref{eq:L-A-eps}, and use the H\"older control estimates from Section~\ref{s:Choice-and-control}. We will do it in the below, in Section~\ref{s:fixing-lengths}, thus proving Proposition~\ref{p:lower-K}.

\subsubsection{Fixing the lengths}\label{s:fixing-lengths}

\begin{proof}[Proof of Proposition~\ref{p:lower-K}]
Let $d=\growth(G/K)$, and $d'=\deg p$, where the polynomial $p$ is given by Lemma~\ref{l:ell}. In order to prove the proposition we will fix a sufficiently small $\eps>0$, make some choice of constants $A_v$ and thus of lengths $L(v,j)$, given by~\eqref{eq:L-A-eps}, and will check that the constructed action is of class $C^{1+\alpha}$.

Note that to do so, it suffices to check only the regularity of the generator elements $g\in \mF$. Also, due to Corollary~\ref{c:split}, it suffices to check that the H\"older constants of their restrictions on intervals~$I_v$ are bounded uniformly in~$v$. 

To do so, we actually first fix the lengths: we fix a sufficiently large constant $C_0>1$ and will choose constants~$A_v$ so that 
\[
L(A_v) = (C_0 + \|v\|)^{-\frac{1}{\alpha}}.
\]

Note that as $\alpha<\frac{1}{d}$, this choice of lengths indeed leads to a convergent series:
\begin{lemma}\label{l:convergence}
The sum $\sum_{v\in X_0} (C_0 + \|v\|)^{-\frac{1}{\alpha}}$ is finite.
\end{lemma}
\begin{proof}
Let us split the sum into parts with $2^{k}\le \|v\| +1 < 2^{k+1}$:
\begin{multline*}
\sum_{v\in X_0} (C_0 + \|v\|)^{-\frac{1}{\alpha}}  = \sum_{k=0}^{\infty} \left( \sum_{2^{k}\le \|v\| +1 < 2^{k+1}} (C_0 + \|v\|)^{-\frac{1}{\alpha}} \right) 
\\
\le 
\sum_{k=0}^{\infty} \left( \#\{v: \|v\|\le 2^{k+1}\}\right) \cdot 2^{-\frac{1}{\alpha}\cdot k} 
\le \sum_{k=0}^{\infty}  C_{G/K} 2^{kd} \cdot 2^{-\frac{k}{\alpha}} = C_{G/K} \cdot \sum_{k=0}^{\infty}  2^{-(\frac{1}{\alpha}-d) \cdot k},
\end{multline*}
and we have obtained an upper bound by a convergent series. 
\end{proof}

Now, for every $v\in X_0$ the restriction of a generator $g\in \mF$ on $I_v$ can be written as 
\[
g = \underbrace{(h_{g(v)} h_{v}^{-1})}_{g_0} \cdot \sigma_v^{\lh(g,v)},
\]
and we will apply Propositions~\ref{p:L-A-Ap} and~\ref{p:sigma-v} to estimate the H\"older constants of the first and the second factor respectively. Namely, as $g$ was a generator, $v$ and $u=g(v)$ are neighbouring vertices in the Schreier graph, hence $\|u\|$ and $\|v\|$ differ at most by~$1$. In particular, if $C_0$ was chosen sufficiently large, the lengths $L(A_v)$ and $L(A_u)$ differ at most twice. Applying Proposition~\ref{p:L-A-Ap}, we get 
\[
\kappa_{\alpha}(g_0, I_v) \le C_{\eps} \cdot \frac{1}{\alpha} \left|\log \frac{C_0+\|u\|}{C_0+\|v\|}\right| \cdot (C_0+\|v\|),
\]
and as $\|v\|-1\le \|u\|\le \|v\|+1$, the absolute value of the logarithm does not exceed $O(\frac{1}{\|v\|})$, and thus the right hand side is bounded uniformly in~$v$.

In order to estimate the H\"older constant related to $\sigma_v^{\lh(g,v)}$, recall that $L(A) \asymp A^{-\eps}$. Hence, as we have $L(A_v)=(C_0 + \|v\|)^{-1/\alpha}$,
\[
A_v \asymp L(A)^{-\frac{1}{\eps}} \asymp (C_0 + \|v\|)^{\frac{1}{\alpha \eps}}.
\]
Choose and fix $\eps>0$ sufficiently small, so that 
\begin{equation}\label{eq:eps}
\frac{1}{\alpha \eps} > d' +1,
\end{equation}
where $d'=\deg p$. Then $|\lh(g,v)|\le p(\|v\|)<A_v$ for all sufficiently large $\|v\|$; choosing $C_0$ to be sufficiently large, we can ensure that the inequality $| \lh(g,v)|<A_v$ holds for all~$v$. Applying Proposition~\ref{p:sigma-v}, we get 
\[
\| D\sigma_v^{\lh(g,v)}  \|_{C(I_v)} \le e^{C_{\sigma}}
\]
and a bound~\eqref{eq:sigma-v-l} for the H\"older constant:
\[
\kappa_{\alpha}(\sigma_v^{\lh(g,v)},I_v) \le C' \cdot \frac{\lh(g,v)/A_v}{ L(A_v)^{\alpha}}  \le C' \frac{p(\|v\|)\cdot  (C_0 + \|v\|)}{A_v} ;
\]
as the numerator is a polynomial of $\|v\|$ of degree $d'+1$, and the denominator has the asymptotics~$\|v\|^{\frac{1}{\alpha \eps}}$, due to~\eqref{eq:eps} the right hand side is bounded uniformly in~$v$.

Finally, 
\[
\kappa_{\alpha}(g, I_v) \le \kappa_{\alpha}(\sigma_v^{\lh(g,v)}, I_v) + \kappa_{\alpha}(g_0, I_v) \cdot \|D\sigma_v^{\lh(g,v)}\|^{\alpha}_{C(I_v)},
\]
and as each of the elements of the right hand side is bounded uniformly in~$v$, so is~$\kappa_{\alpha}(g, I_v)$. The existence of such a uniform bound implies, due to Corollary~\ref{c:split}, that the action of $g$ on all of $I$ is~$C^{1+\alpha}$, thus concluding the proof.
\end{proof}

\begin{remark}
It was stated by K. Parkhe \cite{parkhe} that the critical regularity of a finitely generated torsion free nilpotent group $G$ on the compact interval is greater than $1 + 1/d$,
where $d$ is the exponent of the polynomial growth of $G$. However, as noted by S-h. Kim, N. Matte Bon,  M. de la Salle and M. Triestino~\cite{KMBST},
his proof contains a minor gap: it uses the statement that the spheres of $G$ grow as $n^{d-1}$ in order to prove that the sum of interval lengths 
in the dynamical realization of~$G$ is convergent (see Proposition 2.5 of~\cite{parkhe}), and such an estimate is yet unknown. We would like to note that Parkhe's proof can be easily fixed by summing the lengths in the same way as in Lemma~\ref{l:convergence} above.
\end{remark}

\subsection{Upper bounds for the H\"older constants: technical proofs}\label{s:technical} 

\subsubsection{Auxiliary points estimate lemma}

This section is devoted to the proofs of Propositions~\ref{p:L-A-Ap} and~\ref{p:sigma-v}, as well as of Lemma~\ref{l:A-to-L}. We will be repeatedly using the following easy argument.
\begin{lemma}\label{l:splitting}
Assume that for some map $g$ and for some points $x=x_0,x_1,\dots,x_k=y$ and a constant $c$ one has $|x_j-x_{j-1}|< c|x-y|$. Then 
\[
\frac{|\log Dg(x) - \log Dg(y)|}{|x-y|^{\alpha}} \le c^{\alpha} k \cdot \max_{j=1,\dots,k} \frac{|\log Dg(x_{j}) - \log Dg(x_{j-1})|}{|x_{j}-x_j|^{\alpha}}.
\]
\end{lemma}

In particular, for a map $g$ of the type~\eqref{eq:g-I-v-j} this allows to reduce checking that $\log Dg$ is $\alpha$-H\"older and estimating its constant into two parts: 
\begin{itemize}
\item Checking that for the restrictions $g|_{I_{v,j}}$ the H\"older constants $C^g_j$ of $\log Dg|_{I_{v,j}}$ are uniformly bounded (that can be done by applying the conclusion~\eqref{eq:M-alpha} from Remark~\ref{r:C-R}), and finding an upper bound $C^g$ for these.
\item Checking that restriction $\log Dg|_{\{x_{v,j}\}}$ satisfies the $\alpha$-H\"older condition, and estimating the corresponding H\"older constant.
\end{itemize}

Indeed, for two points $x,y\in I_v$, $x<y$, if they belong to the same subinterval $I_{v,j}$, then one has
\[
\frac{|\log Dg(x) - \log Dg(y)|}{|x-y|^{\alpha}} \le C^g_j \le C^g
\]
while if they belong to two different ones, $x\in I_{v,j}$, $y\in I_{v,j'}$ it suffices to apply Lemma~\ref{l:splitting} using the endpoints $x_1=x_{v,j+1}$ and $x_2=x_{v,j'}$  (provided that $j'>j+1$).

\subsubsection{H\"older estimates for the $l=0$ case}

\begin{proof}[Proof of Proposition~\ref{p:L-A-Ap}]
We start by checking that the $\alpha$-H\"older constants of restrictions  $\log Dg|_{I_{v,j}}$ are bounded uniformly in~$j$. Indeed, the quotient $A_v/A_u$ is uniformly bounded (by some constant $R$ depending on $\eps$) due to the assumption of $\frac{1}{2}<\frac{L(A_v)}{L(A_u)}<2$. Thus, 
\[
\frac{|I_{v,j}|}{|I_{u,j}|} = \left( \frac{A_u^{2} + j^{2}}{A_v^{2} + j^{2}} \right)^{\frac{1+\eps}{2}}
 \in [\frac{1}{R}, R]
\]
as well, and conclusion~\ref{i:bound-M} of Proposition~\ref{p:Tsuboi} is applicable, as well as Remark~\ref{r:C-R}. We then obtain from~\eqref{eq:M-alpha} that the H\"older constant of $\log Dg|_{I_{v,j}}$ does not exceed 
\[
C_R \cdot \frac{|\log Dg(x_{v,j}) - \log Dg(x_{v,j+1})|}{|x_{v,j} - x_{v,j+1}|^{\alpha}}.
\]
In particular, \emph{if} the function $\log Dg$, restricted to a countable set of endpoints 
\[
\{x_{v,j}\mid j\in \Z\} \cup \{\xvl,\xvr\} =  \bd  \cap I_v, 
\]
is $\alpha$-H\"older, then it is $\alpha$-H\"older on full~$I_v$ as well, with at most $(2C_R+1)$ times larger H\"older constant.

Next, in order to check that, note that $Dg(x_{v,-j})=Dg(x_{v,j})$, and the distances $|x_{v,j}-x_{v,j'}|$ and $|x_{v,-j}-x_{v,-j'}|$ differ at most two times for any $j,j'\ge 0$, if suffices to consider only the half of~$I_v\cap X$, that is, the points $x_{v,j}$ with the index $j\ge 0$. Again applying Lemma~\ref{l:splitting} (with the splitting point $x_1=x_{v,[A_v]}$), we can restrict ourselves to the following two cases for $x=x_{v,j}$, $y=x_{v,j'}$: either both $x,y$ belong to the core part, that is, $0\le j<j'\le [A_v]$, or both belong to the outer part, that is, $j'>j\ge [A_v]$. Moreover, this second case can be split into two sub-cases: either $j\le j'\le 2j$ (so that $j$ and $j'$ are comparable), or $j'>2j$, in which case, as we will see below, we can again apply Lemma~\ref{l:splitting} (with $x_1=\xvr$), reducing the estimate to these with $y=\xvr$.

Let us now implement these estimates. To do so, note that for $j\ge 0$ the function 
\[
q(j):= \log Dg|_{x_{v,j}} = \log  \left( \left( \frac{A_u^{2} + j^{2}}{A_v^{2} + j^{2}} \right)^{\frac{1+\eps}{2}} \right)
\]
is actually a \emph{differentiable} function of the variable $j$, even if this variable can take only integer values. Hence, we can (and will) use Lagrange theorem to estimate its increments. More specifically, denote $F(A,j):={\frac{1+\eps}{2}} \log (A^{2}+j^{2})$, so that $q(j)=F(A_v,j)-F(A_u,j)$. Then, from the differentiability of the function $F$ is \emph{both} variables $A$, $j$, applying Lagrange theorem, we get
\begin{multline}\label{eq:q-incr}
q(j)-q(j') = F(A_v,j)-F(A_u,j)-F(A_v,j')+F(A_u,j') 
\\ = (j-j') (A_v-A_u) \cdot \frac{\partial^2 F}{\partial A \, \partial j} (\xi,\eta),
\end{multline}
for some $\xi$ between $A_v$ and $A_u$ and $\eta$ between $j$ and $j'$. Calculating the derivatives, we get
\[
\frac{\partial F}{\partial A} = (1+\eps)   \frac{A}{A^{2}+j^{2}},
\]
\[
\frac{\partial^2 F}{\partial A \, \partial j} = - (1+\eps)   \frac{2jA}{(A^{2}+j^{2})^2}.
\]

Let us now proceed to the consideration of cases we mentioned above: 

\begin{itemize}
\item \textbf{$j,j'\le [A_v]$ (both points in the core part).} 
\\
Note that once $0<j<2A$,
\[
\left|\frac{\partial^2 F}{\partial A \, \partial j} \right| = (1+\eps)   \frac{2jA}{(A^{2}+j^{2})^2} \le \const \cdot \frac{A^{2}}{A^{4}} = \frac{\const}{A^2}.
\]
At the same time, the lengths of all intervals $I_{v,j}$ in the core domain differ at most twice, hence for $0\le j<j'\le [A_v]$ one has
\[
|x_{v,j}-x_{v,j'}| \ge \frac{1}{2} \cdot \frac{j'-j}{A_v^{1+\eps}};
\]
joining it with~\eqref{eq:q-incr}, we get the estimate
\begin{multline*}
\frac{|q(j)-q(j')|}{|x_{v,j}-x_{v,j'}|^{\alpha}} = \left| \frac{\partial^2 F}{\partial A \, \partial j} (\xi,\eta) \right| \cdot \frac{|A_u-A_v| \cdot |j'-j|}{|x_{v,j}-x_{v,j'}|^{\alpha}}  \le 
\\
\le \const \, \frac{|A_u-A_v|}{A_v^2}\cdot \frac{|j'-j|}{|j'-j|^{\alpha}} \cdot A_v^{\alpha(1+\eps)} = \const \cdot \frac{|A_u-A_v|}{A_v} \cdot A_v^{\alpha(1+\eps)-1} \cdot |j'-j|^{1-\alpha}.
\end{multline*}
The right hand side gets takes a maximal value when the difference $j'-j$ is maximal, thus leading to an upper bound by 
\[
\const \cdot \frac{|A_u-A_v|}{A_v} \cdot A_v^{\alpha \eps} \asymp \frac{|\log A_u - \log A_v|}{L(A_v)^{\alpha}}.
\]
Finally, recall that Lemma~\ref{l:A-to-L} guarantees that $\log \frac{L(A_u)}{L(A_v)}\asymp |\log A_u - \log A_v|$, thus we have obtained exactly the upper bound that is claimed in Proposition~\ref{p:L-A-Ap}.

\item \textbf{$[A_v]\le j < j' \le 2j$ (both points in the outer part, sufficiently close to each other).} 
\\
We have an upper bound for the derivative of the function $F$ at $j\ge A$:
\[
\left|\frac{\partial^2 F}{\partial A \, \partial j} \right| = (1+\eps)   \frac{2jA}{(A^{2}+j^{2})^2} \le \const \cdot \frac{jA}{j^{4}} = \const \cdot \frac{A}{j^{3}}.
\]
At the same time, all the intervals $I_{v,i}$ between $x_{v,j}$ and $x_{v,j'}$ are in this case of comparable lengths; in particular, we have 
\[
|x_{v,j}-x_{v,j'}| \ge \frac{j'-j}{(A_v^2 + (j')^2)^{\frac{1+\eps}{2}}} \ge \frac{j'-j}{8 j^{1+\eps}}.
\]

We again apply an estimate by the Lagrange theorem:
\begin{multline*}
\frac{|q(j)-q(j')|}{|x_{v,j}-x_{v,j'}|^{\alpha}} = \left| \frac{\partial^2 F}{\partial A \, \partial j} (\xi,\eta) \right| \cdot \frac{|A_u-A_v| \cdot |j'-j|}{|x_{v,j}-x_{v,j'}|^{\alpha}}  \le 
\\
\le \const \, \frac{A_v}{j^{3}} \, |A_u-A_v| \frac{|j'-j|}{|j'-j|^{\alpha}} \cdot j^{\alpha(1+\eps)}.
\end{multline*}
For every fixed $j$, the expression in the right hand side takes its maximal value when $j'-j$ does, that is, for $j'=2j$. This upper bound is then
\[
\const \, \frac{A_v^2}{j^{3}} \, \frac{|A_u-A_v|}{A_v} \cdot j^{(1-\alpha)} \cdot j^{\alpha(1+\eps)} = 
\const \, \frac{A_v^2}{j^{2-\alpha\eps}} \, \frac{|A_u-A_v|}{A_v}.
\]
As the power for $j$ in the denominator is positive, $2-\alpha\eps>0$, the right hand side of this upper bound takes the least value for $j=[A_v]$, thus leading to an upper bound by 
\[
\const \, \frac{A_v^2}{A_v^{2}\cdot A_v^{-\alpha\eps}} \, \frac{|A_u-A_v|}{A_v} \asymp \frac{|\log A_u-\log A_v|}{L(A_v)^{\eps}};
\]
again, the numerator of this estimate is comparable to the desired $|\log L(A_u)-\log L(A_v)|$ due to Lemma~\ref{l:A-to-L}.
\item \textbf{$j'\ge 2j$ and $j\ge [A_v]$ (both points in the outer part, sufficiently far away from each other)}. 
\\
Consider first a degenerate version of this case as $j'\to\infty$, corresponding to taking $y=\xvr$. We then have $\log Dg(\xvr)=0$; by Lagrange theorem for $q(j)=F(A_v,j)-F(A_u,j)$, we have that for some $\xi$ between $A_u$ and $A_v$
\[
|q(j)| = \left|(A_v-A_u) \cdot \frac{\partial F}{\partial A}|_{(\xi,j)} \right| \asymp |A_v-A_u| \cdot \frac{A_v}{j^{2}};
\]
at the same time, the corresponding distance can be bounded from below by 
\begin{equation}\label{eq:j-upper}
|x_{v,j}-\xvr| = \sum_{r=j}^{\infty} \frac{1}{(A_v^{2}+j^{2})^{\frac{1+\eps}{2}}} \ge \const \, \int_r^{\infty} \frac{dr}{r^{1+\eps}} \asymp \frac{1}{j^{\eps}}.
\end{equation}
We thus get 
\[
\frac{|\log Dg(x_{v,j})-\log Dg(\xvr)|}{|x_{v,j}-\xvr|^{\alpha}} \le \const \frac{|A_u-A_v|}{A_v} \, \frac{A_v^{2}}{j^{2}}\cdot j^{\alpha \eps}
\]
Again, as $1+\eps>\alpha\eps$, the right hand side takes the maximal value at the minimal value of $j$, that is, $j=[A_v]$, leading to an upper bound
\begin{equation}\label{eq:j-to-end}
\frac{|\log Dg(x_{v,j})-\log Dg(\xvr)|}{|x_{v,j}-\xvr|^{\alpha}} \le  \const \frac{|A_u-A_v|}{A_v} \, \frac{A_v^{2}}{A_v^{2}}\cdot A_v^{\alpha \eps} \asymp \frac{|\log L(A_u)- \log L(A_v)|}{L(A_v)^{\alpha}},
\end{equation}
where we have again applied Lemma~\ref{l:A-to-L} both for the numerator and for the denominator.

Finally, for general $j$ and $j'$ with $j'\ge 2j$, note that the distances $|x_{v,j}-x_{v,j'}|$ and $|x_{v,j}-\xvr|$ are comparable. Indeed, the former distance is bounded from below by 
\[
|x_{v,j}-x_{v,j'}|= \sum_{r=j}^{j'-1} \frac{1}{(A_v^{2}+j^{2})^{\frac{1+\eps}{2}}} \ge \sum_{r=j}^{2j-1} \frac{1}{2\cdot (2j)^{1+\eps}} \ge  \frac{j}{8 j^{1+\eps}} = \frac{1}{8 j^{\eps}};
\]
comparing it to~\eqref{eq:j-upper} provides the desired uniform estimate: there exists $c=c(\eps)$ such that for all $j'\ge 2j$, $j\ge [A_v]$ 
\[
|x_{v,j}-\xvr| \le c \cdot |x_{v,j}-x_{v,j'}|.
\]
Now, an application of Lemma~\ref{l:splitting} with 
\[
x_0=x=x_{v,j}, \quad x_1=\xvr, \quad x_2=y=x_{v,j'}
\] 
allows to extend the upper bound~\eqref{eq:j-to-end} to general $j'\ge 2j$.
\end{itemize}
In all the three cases we have obtained the desired upper bound, thus completing the proof of the proposition.
\end{proof}

\subsubsection{H\"older estimates for the $u=v$ case}

\begin{proof}[Proof of Proposition~\ref{p:sigma-v}] 
We start by establishing the upper bound for the derivatives. Let us establish it first at the set of endpoints $\bd \cap I_v$, where
\begin{equation}\label{eq:D-x-v-j}
\log D \sigma_v^l (x_{v,j}) = \log \frac{|I_{v,j+l}|}{|I_{v,j}|}.
\end{equation}
At such points, if $|j|\le 2A_v$ (core part), the fraction in the right hand side is bounded due to~\eqref{eq:length-core}, and for $|j|>2A_v$, it is bounded due to~\eqref{eq:length-outer}. Next, as the quotient of lengths $\frac{|I_{v,j+1}|}{|I_{v,j}|}$ is uniformly bounded, the estimate~\eqref{eq:D-x-v-j} extends inside each interval $I_{v,j}$ due to the uniform bound on the distortion of the map $\sigma_v|_{I_{v,j}}= \varphi_{I_{v,j-1},I_{v,j}}^{I_{v,j+l-1},I_{v,j+l}}$, see estimate~\eqref{eq:M} in Proposition~\ref{p:Tsuboi}.

Now, note that due to these upper bounds for the derivative, it suffices to prove~\eqref{eq:sigma-v-l} for $l=1$. Indeed, 
assume that we had established
\begin{equation}\label{eq:C-1}
\kappa_{\alpha}(\sigma_v) \le C'_1 \cdot \frac{1/A_v}{ L(A_v)^{\alpha}}.
\end{equation}
Then for a general $l>0$, applying the chain rule~\eqref{eq:kappa-composition}, we get
\[
\kappa_{\alpha}(\sigma_v^l) \le  \sum_{r=1}^{l} \kappa_{\alpha}(\sigma_v) \cdot \| D \sigma_v^{r-1}\|^{\alpha}_{C(I_v)} \le e^{\alpha C_{\sigma}} \cdot l \kappa_{\alpha}(\sigma_v),
\]
and thus the estimate~\eqref{eq:sigma-v-l} is satisfied with $C':=C'_1 \cdot e^{\alpha C_{\sigma}}$.

To obtain the upper bound~\eqref{eq:C-1}, we will proceed in the same way as in the proof of Proposition~\ref{p:L-A-Ap}. In the same way as before, due to Lemma~\ref{l:splitting} it suffices to obtain a bound for $\frac{|\log D \sigma_v(x)-\log D \sigma_v(y)|}{|x-y|^{\alpha}}$ for $x,y\in \bd\cap I_v$. Again in the same way as before, due to the symmetry we restrict the case $x=x_{v,j}$, $y=x_{v,j'}$ to $j,j,'\ge 0$. Denoting 
\[
\tq(j):= \log D\sigma_v (x_{v,j}) = F(A_v,j)-F(A_v,j+1),
\]
we get 
\begin{multline}\label{eq:t-q-incr}
\tq(j)-\tq(j') = F(A_v,j)-F(A_v,j+1)-F(A_v,j')+F(A_v,j'+1) 
\\ = (j-j')  \cdot \frac{\partial^2 F}{\partial j^2 } (A_v,\eta),
\end{multline}
for some $\eta$ between $j$ and $j'$. 
From an explicit computation,
\[
\frac{\partial F}{\partial j} = (1+\eps)   \frac{j}{A^{2}+j^{2}},
\]
\[
\frac{\partial^2 F}{\partial j^2} = - (1+\eps)    \frac{A^2-j^2}{(A^{2}+j^{2})^2}.
\]

Now, we split the general case into three cases: \\
\begin{itemize}
\item \textbf{$0\le j<j' \le [A_v]$ (both points in the core part).} 
\\
In this domain, 
\[
|\frac{\partial^2 F}{\partial j^2} | = (1+\eps)    \frac{|A_v^2-j^2|}{(A_v^{2}+j^{2})^2}. \le \frac{1+\eps}{A_v^2},
\]
hence we have 
\[
|\log D\sigma_v(x_{v,j})-\log D\sigma_v(x_{v,j'})| = |j-j'| \cdot |\frac{\partial^2 F}{\partial j^2 } (A_v,\eta)| \le (1+\eps) \frac{|j-j'|}{A_v^2}.
\]
Thus, 
\[
\frac{|\log D\sigma_v(x_{v,j})-\log D\sigma_v(x_{v,j'})|}{|x_{v,j}-x_{v,j'}|^{\alpha}} \le \const  \frac{|j-j'|/A_v^2}{|j-j'|^{\alpha}/A_v^{\alpha(1+\eps)}}
\]
The value of the right hand side is maximal if $j=0$, $j'=[A_v]$, when it is equal to  
\[
\const \frac{A_v/A_v^2}{ A_v^{\alpha}/A_v^{\alpha(1+\eps)}} \asymp \frac{1}{A_v L(A_v)^{\alpha}},
\]
which is exactly the desired upper bound.
\item  \textbf{$[A_v]\le j < j' \le 2j$ (both points in the outer part, sufficiently close to each other).} 
\\
In this regime
\[
|\frac{\partial^2 F}{\partial j^2} | = (1+\eps)    \frac{|A_v^2-j^2|}{(A_v^{2}+j^{2})^2}. \le \frac{1+\eps}{j^2}
\]
and 
\[
|x_{v,j}-x_{v,j'}| \asymp \frac{|j-j'|}{j^{1+\eps}}.
\]
Thus, 
\[
\frac{|\log D\sigma_v(x_{v,j})-\log D\sigma_v(x_{v,j'})|}{|x_{v,j}-x_{v,j'}|^{\alpha}} \le \const \,\frac{|j-j'|/j^2}{|j-j'|^{\alpha}/j^{\alpha(1+\eps)}}.
\]
For any given $j\ge [A_v]$, the right hand side takes maximal value when $j'-j$ is maximal, that is, for $j'=2j$. In this case, the right hand side is equal to 
\[
 \const \,\frac{j/j^2}{j^{\alpha}/j^{\alpha(1+\eps)}} =\const \frac{1}{j /j^{\alpha\eps}}
\]
In turn, this expression is maximal when $j$ takes the least possible value in the domain, $j=[A_v]$. Then it is equal to
\[
\const \frac{1}{[A_v] /[A_v]^{\alpha\eps}} \asymp \frac{1}{A_v L(A_v)^{\alpha}},
\]
again obtaining exactly the desired estimate.

\item \textbf{$j'\ge 2j$ and $j\ge [A_v]$ (both points in the outer part, sufficiently far away from each other)}. 
\\

Consider first a degenerate version of this case as $j'\to\infty$, corresponding to taking $y=\xvr$. We then have $\log Dg(\xvr)=0$; by Lagrange theorem for $\tq(j)=F(A_v,j)-F(A_v,j+1)$, we have that for some $\eta\in [j,j+1]$ one has
\[
|\tq(j)| = \left|  \frac{\partial F}{\partial j}|_{(A_v,\eta)} \right| \asymp \frac{1}{j}.
\]
Comparing it with a lower bound~\eqref{eq:j-upper} for the distance between the points $x_{v,j}$ and $x_{v,j'}$, we get 
\[
\frac{|\log D\sigma_v(x_{v,j})-\log D\sigma_v(\xvr)|}{|x_{v,j}-\xvr|^{\alpha}} \le \const \frac{1/j}{(1/j^{\eps})^{\alpha}} \le \const \frac{1/A_v}{(1/A_v^{\eps})^{\alpha}} \asymp \frac{1}{A_v L(A_v)^{\alpha}}; 
\]
here we have used that, as $1+\eps>\alpha\eps$, the right hand side takes the maximal value at the minimal value of $j$, that is, $j=[A_v]$.
Finally, the general case of $j'\ge 2j$, $j\ge [A_v]$ is handled by applying Lemma~\ref{l:splitting} (with $x_1=x_v^+$).
\end{itemize}

\end{proof}

\begin{proof}[Proof of Lemma~\ref{l:A-to-L}]
The first part has been already proven earlier by adding the estimates~\eqref{eq:A-core} and~\eqref{eq:A-outer} for the lengths of the core and outer parts respectively. To prove the second one, note that the total length $L(A)$ is actually a differentiable function of~$A$: bringing the derivative under the sign of the summation, we get 
\begin{equation}\label{eq:dL-dA}
- \frac{1}{1+\eps} \frac{d}{dA} L(A) = \sum_{j=-\infty}^{\infty} \frac{A}{(A^2+j^2)^{\frac{3+\eps}{2}}}.
\end{equation}
Note that the sum in the right hand side converges uniformly for $A$ belonging to any compact subset of $[1,+\infty)$.

Now, we have 
\[
\sum_{|j|\le A} A \cdot (A^2+j^2)^{-\frac{3+\eps}{2}} \asymp (2A+1) \cdot A^{-2-\eps} \asymp A^{-1-\eps}
\]
and 
\[
\sum_{|j|> A} A \cdot (A^2+j^2)^{-\frac{3+\eps}{2}} \asymp A \cdot \int_A^{\infty} j^{-2-\eps} \, dj \asymp A^{-1-\eps}.
\]
Adding these together, we see that 
\[
- \frac{1}{1+\eps} \frac{d}{dA} L(A) \asymp A^{-1-\eps},
\]
thus implying 
\[
- \frac{1}{1+\eps} \frac{d}{dA} \log L(A) \asymp \frac{A^{-1-\eps}}{A^{-\eps}} = A^{-1} = \frac{d}{dA} \log A.
\]
Integration over $A$ concludes the proof of the lemma.
\end{proof}

\section{Completing the proofs of the main theorems}\label{s:main-proofs}

In this section, we complete the proofs of our main results: Theorems~\ref{t:A}, \ref{thm crit interval}, \ref{t:half-open} and~\ref{t:circle}.

\begin{proof}[Proof of Theorem~\ref{t:A}]
Let $G$ be a torsion-free nilpotent finitely generated group, and $c\in Z(G)$ its non-trivial central element. Consider any action $\phi:G\to \Diff_+^{1+\alpha}(M)$, for some $\alpha>0$, where $M=[0,1]$ or $(0,1]$, such that $\phi(c)\neq \id$. Then there is a non-trivial support interval $I_c$ of $\phi(c)$; let $K:=\Stab(I_c)$ be the corresponding stabilizer subgroup. Then $K\in \StSub(G,c)$ by Proposition~\ref{p:stab-subgr}, and by Propositions~\ref{thm upper bound},~\ref{p:critical} we have $\alpha<1/\growth(G/K)$.
On the other hand, for any $K\in \StSub(G,c)$ and any $\alpha<1/\growth(G/K)$, by Proposition~\ref{p:lower-K} there exists an action $\phi\in \Class_{[0,1]}^0(G,c)$ of regularity~$C^{1+\alpha}$. Taking the maximum over all possible $K\in\StSub(G,c)$ completes the proof of the theorem for the cases $M=[0,1]$ and $M=(0,1]$. (Note that the growth of $G/K$ is an integer number, so it can take only a finite set of possible values, and the maximum is thus attained.)

It is clear that the actions constructed in Proposition~\ref{p:lower-K} can be defined on $\mathbb{S}^1$ by placing a global fixed point on it. Therefore, one might expect the critical regularity on $\mathbb{S}^1$ to be greater than that on $[0,1]$. However, as an immediate consequence of the following lemma, this is not the case; the critical regularity on both manifolds actually coincides.
\begin{lemma}\label{l:circle}
Let $G \leqslant \Homeo_+(\mathbb{S}^1)$ be a nilpotent torsion-free group, and let $c \in Z(G)\setminus \{e\}$ be a non-trivial central element. The following assertions hold:
\begin{enumerate}
\item\label{i:irrational} If no power of $c$ has fixed points on $\mathbb{S}^1$, then $G \in \StSub(G,c)$.
\item\label{i:rational} Suppose that $c^k$ has fixed points on $\mathbb{S}^1$ for some $k \in \mathbb{Z}$. Then, for any support interval $I_{c^k}$ of~$c^k$, there exists $K \in \StSub(G,c)$ such that
$K \supseteq \Stab(I_{c^k}).$
\end{enumerate}
\end{lemma}
\begin{proof} 
Note first that in this case the rotation number $\rho:G\to \R/\Z$ is a group homomorphism. Indeed, as the group $G$ is nilpotent and is acting on a compact set $M=\Sc$, there exists a measure~$\nu$ that is invariant for all $f\in G$. Next, for any $f\in G$ and any $x\in \Sc$ one has $\rho(f)=\nu([x,f(x))) \mod \Z$, what easily implies that $\rho(gf)=\rho(g)+\rho(f)$ (it suffices to use $y=f(x)$ to calculate $\rho(g)$.

Let us now establish conclusion~\ref{i:irrational}. Consider the image group $\rho(G)$, that is a subgroup of $\Sc=\R/\Z$. It is a finitely generated abelian group, and $\rho(c)\in \rho(G)$ is not a torsion element therein: $c$ has no periodic points, and hence its rotation number is irrational. Take a quotient of $\rho(G)$ by its torsion subgroup~$T$; the result is isomorphic to some~$\Z^k$, and the image of $c$ therein is nontrivial, thus there exists a (rank $k-1$) subgroup $H_1$ such that the quotient $(\rho(G)/T)/H_1\simeq \Z$ and that the image of $c$ in this quotient is non-trivial. We then take $H:=\rho^{-1}(T+H_1)$, thus obtaining $G/H\simeq \Z$ and $H\cap \langle c\rangle=\{e\}$. Thus, the whole group $G$ is a stabilizer subgroup of $c$, where the ascending chain of subgroups appearing in Definition~\ref{def stab subg} is trivial.

Next, let us to establish the conclusion~\ref{i:rational} of the lemma. To do it, we will use Remark~\ref{rq:Zr} and construct the chain~\eqref{eq:Kp-chain}, starting from $K_{(0)}=G$, and proceeding in such a way that on every step~$i$ one has the inclusion $\Stab(I_{c^k})<K_{(i)}$. We start by eliminating the irrational rotation numbers: let $T_1:=\Tor(\rho(G))$ be the torsion subgroup, then $\rho(G)/T_1 \simeq \Z^{r_0}$ is an abelian finitely generated torsion-free group, thus it is isomorphic to some $\Z^{r_0}$. Hence, taking $K_{(1)}:=\rho^{-1}(T_1)$ we ensure $K_{(0)}/K_{(1)}\simeq \Z^{r_0}$ and $\rho(K_{(1)})=T_1$.

Now, for every given $i$ note that $T_i=\rho(K_{(i)})$ is a finite subgroup of $\Sc$, hence it is generated by $\rho(g_i)$ some $g_i\in K_{(i)}$. Also, denote by $K_{(i)}':=K_{(i)}\cap \Ker(\rho)$ the kernel of the rotation number homomorphism, restricted on $K_{(i)}$, and let $X_i$ be the set of common fixed points of the action of~$K_{(i)}'$ (note that $X_i$ contains the support of any invariant measure of this action, thus the set~$X_i$ is nonempty). Let the interval $J_i$ be the closure of the connected component of $\Sc\setminus X_i$ that contains~$I_{c^k}$. Consider its $K_{(i)}$-orbit (see Fig.~\ref{fig:circle-intervals}): let 
\[
J_{i,j}=g_i^j (J_i), \quad j=0,1,\dots, q_i-1, \quad \text{where } q_i:=|T_i|.
\]
\begin{figure}[!h!]
\begin{center}
\includegraphics{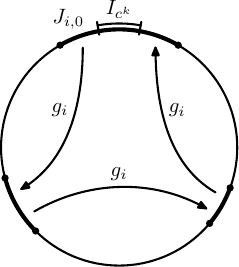}
\end{center}
\caption{Bold points: global fixed points of the action of $K_{(i)}'$; the interval $J_{i,0}$ between such points, containing the support interval $I_{c^k}$, and its orbit under the iterations of the map $g_i$.}\label{fig:circle-intervals}
\end{figure}

Let $\mu_{i,0}$ be a Radon invariant measure for the action of $K_{(i)}'$ on $J_i=J_{i,0}$, and let measures $\mu_{i,j}$ on $J_{i,j}$ be its pushfowards,
\begin{equation}\label{eq:mu-i-j}
\mu_{i,j}=(g_i)_*^j (\mu_{i,0}), \quad j=0,1,\dots, q_i-1.
\end{equation}

Associated with each pair $(I_{i,j}, \mu_{i,j})$, we have a translation morphism $\tau_{i,j}: K_{(i)}' \to \mathbb{R}$; the pushforward construction~\eqref{eq:mu-i-j} implies that
\[
\tau_{i,j}(f) = \tau_{i,j+1}(g f g^{-1}) \; \text{for all} \; j \in \mathbb{Z}/q_i\mathbb{Z}\;\text{and all}\; f \in K_{(i)}'.
\]
In particular, one then has 
\begin{equation}\label{eq:tau-comm}
\tau_{i,j}([f,g]) = \tau_{i,j}(f) + \tau_{i,j} (gf^{-1} g^{-1} ) = \tau_{i,j}(f) - \tau_{i,j-1}(f).
\end{equation}
We claim that~\eqref{eq:tau-comm} implies that 
\begin{equation}\label{eq: equality of transl numbers}
\tau_{i,j}(f) = \tau_{i,j'}(f)\;\text{for all}\; j, j' \in {0, \ldots, q_i-1}\;\text{and all}\;f \in K_{(i)}'.
\end{equation} 
Indeed, if the coordinates of a vector $v=(v_1,v_2,\dots,v_n)$ are not equal, then the same is true for the vector of cyclic differences $(v_1-v_n,v_2-v_1,\dots,v_n-v_{n-1})$ (the sum of differences $v_i-v_{i-1}$ is equal to zero, hence there are both positive and negative coordinates). Thus, if~\eqref{eq: equality of transl numbers} did not hold for some $f\in \Ker(\rho)$, then it would not hold for $f_1:=[f,g]$ as well, nor for $f_2:=[f_1,g]=[[f,g],g]$, etc. We thus would obtain an infinite chain of commutators that are all non-trivial (for $f=e$ the equality~\eqref{eq: equality of transl numbers} is trivially satisfied), and this would be a contradiction with the nilpotence of the group~$G$.

Due to~\eqref{eq: equality of transl numbers}, 
\begin{equation}\label{eq:conj-inv}
\tau_{i,0}(f)=\tau_{i,0}(g_i^{-j} f g_i^j) \quad \forall f\in K_{(i)}', \quad \forall j\in\Z.
\end{equation}
Let us extend the translation number $\tau_{i,0}:K_{(i)}'\to \R$ to a map $t_i:K_{(i)}\to \R$ by letting
\begin{equation}
t_i(f g_i^j):=\tau_{i,0}(f) + \frac{j}{q_i} \tau_{i,0}(g_i^{q_i}) \quad \forall f\in K_{(i)}', \,\, j\in\Z.
\end{equation}
It is easy to see that this extension is well-defined (choosing a different representation of the argument by changing~$j$ by $q_i$ does not change the result); also, due to~\eqref{eq:conj-inv} the extended map $\tau_i$ is still a homomorphism: 
\begin{multline*}
t_i (f g_i^j \cdot f' g_i^{j'}) = t_i (f \cdot (g_i^j f' g_i^{-j} ) \cdot g_i^j g_i^{j'}) = \tau_{i,0} (f) + \tau_{i,0} (g_i^j f' g_i^{-j} ) + \frac{j+j'}{q_i} \tau_{i,0}(g_i^{q_i}) 
\\ 
= \tau_{i,0} (f) + \frac{j}{q_i} \tau_{i,0}(g_i^{q_i}) + \tau_{i,0} (f' ) + \frac{j'}{q_i} \tau_{i,0}(g_i^{q_i}) = t_i (f' g_i^{j'})+ t_i (f' g_i^{j'}).
\end{multline*}
In particular, this extension can be equivalently defined by taking
\begin{equation}\label{eq: transl morphism}
t_i(f) := \frac{\tau_{i,0}(f^{q_i})}{q_i}.
\end{equation}
Having defined $t_i$, we proceed in the following way. If $I_{c^k}$ was a proper subset of $J_i$, we define $K_{(i+1)}:=\Ker t_i$. Then, any map $f\in \Stab(I_{c^k})$ then has a fixed point inside $J_i$, and thus belongs to this kernel, ensuring the inclusion $\Stab(I_{c^k})<K_{(i+1)}$. On the other hand, $K_{(i)}/K_{(i+1)}\simeq \Im t_i$, and as $\Im t_i$ is a subgroup of~$\R$, it is isomorphic to some~$\Z^{r_i}$.

Otherwise, we stop the construction of the chain, letting $K:=K_{(i)}$, and construct the corresponding subgroup $H$. To do so, note that as $I_{c^k}=J_i$, the map $c^k$ has no fixed points inside $J_i$, and hence $t_i(c)\neq 0$. The image $\Im t_i$ is isomorphic to some $\Z^l$; let $\psi:\Im t_i \to \Z^l$ be the corresponding isomorphism, and denote $v_c:=\psi(t_i(c))$. In the same way as in the proof of Proposition~\ref{p:stab-subgr}, $v_c\neq 0$, hence there exists a subgroup~$\tH<\Z^l$ (actually, one of the coordinate hyperplanes), intersecting trivially $\langle v_c \rangle$ and such that $\Z^l/\tH\simeq \Z$. We take $H:=t_i^{-1} (\psi^{-1}(\tH))$; then, $K/H\simeq \Z$ and $H\cap \langle c \rangle =\{e\}$, and the condition~\eqref{i:H} of the stabilizer subgroup is thus verified.

To complete the proof of the conclusion~\ref{i:rational}, note that the above construction of the chain~\eqref{eq:Kp-chain} stops in a finite number of steps. Indeed, if it does not stop after $i$ steps, the degree of polynomial growth of the group $G$ is greater or equal to~$i$, hence the construction stops in at most $\growth (G)$ steps.
\end{proof}

Lemma~\ref{l:circle} allows to conclude the proof of Theorem~\ref{t:A} for the case of~$M=\Sc$. Indeed, let $c\in Z(G)\setminus \{e\}$, and consider any action $\phi\in \Class_{\Sc}(G,c)$. If $c$ acts with an irrational rotation number, we are in the case~\ref{i:irrational} of the lemma, and thus there is a $C^{\infty}$-action $\psi\in\Crit \Class_{[0,1]}^0 (G,c)$ thus both $\Crit \Class_M(G,c)$ and $\Crit \Class_{[0,1]}^0 (G,c)$ are infinite. Otherwise, we are in the case~\ref{i:rational} of the lemma. If the action $\phi$ is of regularity $C^{1+\alpha}$, and $k\in \N$, $K\in\StSub(G,c)$ are as in the conclusion of the lemma, using Proposition~\ref{thm upper bound} we have 
\[
\alpha< \frac{1}{\growth (G/\Stab(I_{c^k}))} \le \frac{1}{\growth (G/K))}.
\]
This implies the inequality
\[
\Crit \Class^0_{[0,1]}(G,c) \geq \Crit \Class_{\Sc}(G,c);
\]
together with the inequality in the other direction (glueing the endpoints of the interval to obtain an action on the circle with a global fixed point) this completes the proof of Theorem~\ref{t:A}.
\end{proof}

\begin{proof}[Proof of Theorems~\ref{thm crit interval},~\ref{t:half-open} and~\ref{t:circle}]
Assume that for some $\alpha>0$ we are given a faithful action $\phi:G\to \Diff_+^{1+\alpha}(M)$, where $M=[0,1]$, $(0,1]$ or $\Sc$. Then for each $c\in Z(G)\setminus \{e\}$ the element $c$ acts non-trivially, thus the upper bound from Theorem~\ref{t:A} implies that 
\[
\alpha < \max_{K\in \StSub(G,c)} \frac{1}{\growth(G/K)}.
\]
Taking a minimum of such upper estimates over all possible $c$ provides the desired upper bound for~$\alpha$, 
\begin{equation}\label{eq:alpha}
\alpha<\left[ \max_{c\in Z(G)\setminus \{e\}} \min \{\growth(G/K) \, : \,  K \in \StSub(G,c)\} \right]^{-1}.
\end{equation}

In the other direction, assume that for $\alpha$ the inequality~\eqref{eq:alpha} holds. Take an arbitrary nontrivial $c_1\in Z(G)\setminus \{e\}$. Then by Theorem~\eqref{t:A} there exists an action $\phi_1:G\to\Diff^{1+\alpha}_+(J_1)$ of the group~$G$ on some interval $J_1$, tangent to the identity at the endpoints of this interval, such that $\phi_1(c_1)\neq\id$. Consider $\Ker \phi_1 \cap Z(G)$. If this intersection is trivial, the action is faithful by Lemma~\ref{l:on-centre}. Otherwise, take any non-trivial $c_2 \in \Ker \phi_1 \cap Z(G)$. Again, by Theorem~\eqref{t:A} there exists an action $\phi_2:G\to\Diff^{1+\alpha}_+(J_2)$ of the group~$G$ on some interval $J_2$, tangent to the identity at the endpoints of this interval, such that $\phi_2(c_2)\neq\id$. We can assume that the intervals $J_1$ and $J_2$ are adjacent to each other, and as the action is tangent to the identity at their common endpoint, the resulting action on $J_1\cup J_2$ is also of the class $C^{1+\alpha}$. 

\begin{figure}[!h!]
\begin{center}
\includegraphics{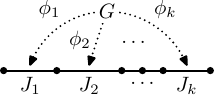}
\end{center}
\caption{Joining the actions $\phi_i$.}\label{fig:joining}
\end{figure}

Now, either $\Ker \phi_1 \cap \Ker \phi_2 \cap Z(G) =\{e\}$, in which case the required action is constructed, or not, in which case we consider any nontrivial element $c_3$ of this intersection and construct the associated action~$\phi_3$ on some interval $J_3$ that we presume to be adjacent to~$J_2$. We continue this process of joining actions (see Fig.~\ref{fig:joining}); while doing so, the rank of the abelian torsion-free group $\left(\bigcap_i \Ker \phi_i\right) \cap Z(G)$ decreases at least by~$1$ on every step. Hence, in a finite number of steps this process stops, providing us with the desired faithful action $\phi$ of the group $G$ on the interval $J=\bigcup_i J_i$.

Finally, an orientation-preserving action on the interval $J=[0,1]$ that is tangent to the identity at its endpoints can be transformed both into an action on the half-open interval $(0,1]$ (by removing the endpoint~$x=0$) and into an action on the circle $\Sc=[0,1]/(0\sim 1)$ by glueing its endpoints.
\end{proof}

\begin{proof}[Proof of Remark~\ref{rq:circle-torsion}]
We will first prove that the set $T$ of torsion elements of the group $G$ is a finite cyclic subgroup; the proof here largely follows the proof of Lemma~\ref{l:circle}, as we actually did not use torsion-freeness of the group $G$ there. Namely, we start by passing from $G$ to $K_{(1)}=\rho^{-1}(T_1)$, where $\rho:G \to \R/\Z$ is the rotation number homomorphism (associated to some $G$-invariant measure on the circle), and $T_1$ is the torsion subgroup of the image $\rho(G)$. Indeed, $\rho(T)\subset T_1$, thus implying $T\subset \rho^{-1}(T_1)$. 

Again in the same way as in the proof of Lemma~\ref{l:circle}, we construct the sequence of the subgroups $K_{(1)}>K_{(2)}>\dots$. Namely, we take $T_i=\rho(K_{(i)})$ the set of possible values of the rotation numbers in $K_{(i)}$, an element $g_i\in K_{(i)}$ such that $\rho(g_i)$ generates $T_i$. For the subgroup $K_{(i)}':=K_{(i)} \cap \Ker \rho$ the set of its common fixed points is non-empty, as it includes the support of any $K'_{(i)}$-invariant measure; we take any interval $J_{i,0}$ that is a connected component of $\Sc\setminus \Fix(K'_{(i)})$, and take a Radon $K'_{(i)}$-invariant measure $\mu_{i,0}$ therein. 

We use this measure to define the translation number $t_i:K'_{(i)}\to \R$ and extend it to $t_i:K_{(i)}\to \R$ using~\eqref{eq:conj-inv} (or, equivalently,~\eqref{eq: transl morphism}), and then define the next subgroup in the chain as $K_{(i+1)}:=\Ker t_i$. Note, that torsion elements have zero translation number, so for every $i$ we have $T\subset K_{(i)}$. This process continues until at some moment $r$ we get $\Fix(K'_{(r)})=\Sc$, thus implying $K'_{(r)}=\{\id\}$, at which moment we get $K_{(r)}=T$ with the rotation number becoming a homomorphism $\rho:T\to T_r$. This implies that $T$ is a cyclic subgroup of $G$.

Now, the set $T$ of torsion elements is invariant under conjugacies, and as a conjugacy by any $g\in G$ preserves the rotation number, it fixes the elements of $T$ pointwise, hence $T$ is contained in the center of~$G$. Finally, the quotient group $G/T$ is torsion-free (otherwise the preimage of a nontrivial torsion element in $G/T$ would be a torsion element in $G\setminus T$). If the action of $G$ is of class $C^{1+\alpha}$, it is $C^{1+\alpha}$-conjugate to an action with the Lebesgue measure that is $T$-invariant (it suffices to average the Lebesgue measure over the action of $T$ to obtain an invariant measure with $C^{\alpha}$-density). This action is thus also of class $C^{1+\alpha}$; it descends to a $C^{1+\alpha}$-action of the quotient $G/T$ on the quotient circle $\Sc/T$. As $G/T$ is torsion-free, together with Theorem~\ref{t:circle} this completes the proof of~\eqref{eq:G/T}.
\end{proof}

 \section{Concluding remarks: other classes of actions}\label{s:examples}
 
 We conclude this work by presenting two examples where critical regularity (for other classes of actions) differs from the one predicted by Theorems~\ref{thm crit interval} and~\ref{t:half-open}.
 
 \subsection{Actions on the open interval}\label{s:open-interval}
 
 Consider the actions on the open interval (or, what is the same for the regularity reasons, on the real line). The following example shows that the critical regularity of a nilpotent group can be higher than for the actions on the closed interval. Recall that the group $N_d$ is the group of $d\times d$ upper-triangular matrices with integer elements and $1$'s on the main diagonal.
 
 \begin{example}\label{ex-N4}
For any $\alpha<1$ there exists $C^{1+\alpha}$-action of the group $N_4$ on the real line. Hence,
\begin{equation}\label{eq:crit-N4}
\Crit_{(0,1)}(N_4)=\Crit_{\R}(N_4) \ge 2.
\end{equation}
 \end{example}
 In particular, the regularity in~\eqref{eq:crit-N4} is strictly larger, than the regularity $\Crit_{[0,1]}(N_4)=1+1/2$, found by E.~Jorquera, A.~Navas, C.~Rivas in~\cite{JNR}.
 
 \begin{proof}
$N_4$ is a subgroup of $\SL_{4}(\Z)$, formed by the matrices of the form
 \[
\left(
\begin{matrix}
1 & a_{12} & a_{13} & a_{14}\\
0 & 1 & a_{23} & a_{24}\\
0 & 0 & 1 & a_{34}\\
0 & 0 & 0 & 1 
\end{matrix}
\right), \quad \forall i<j \quad a_{ij}\in \Z.
 \]
 It has a cyclic center $Z(N_4)=\langle c \rangle$, generated by the matrix $c$ whose only nonzero above-diagonal element is $a_{14}=1$. Also, consider the subgroups
 \[
H<K<G<N_4,
 \]
 defined by 
 \[
G=\{A \in N_4 \mid a_{12}=0\}, \quad K=\{A \in G \mid a_{34}=0\}, \quad H=\{A\in K \mid a_{14}=0\}.
 \]
 Then $K\simeq \Z^4$ is a stabilizer subgroup of an element $c\in G$, as $H\cap \langle c \rangle = \{e\}$ and $K/H \simeq \Z$, with the relative growth $\growth(G/K)=1$. By Proposition~\ref{p:lower-K}, for every $\alpha$ there exists an action $\phi$ of $G$ on~$[0,1]$ by $C^{1+\alpha}$-diffeomorphisms, tangent to the identity at the endpoints, such that the action of~$c$ is non-trivial. 
 
Next, $G\triangleleft N_4$ and $N_4/G\simeq \langle b \rangle$, where for the matrix $b$ its only nonzero above-diagonal element is $a_{12}=1$. Then we can extend $\phi$ to an action of $N_4$ on $\R$ by taking $\phi(b)(x)=x+1$. Namely, we first extend the action of $G$ on every interval $[n,n+1]$ by setting
 \[
\phi(g) (x+n) = \phi(b^{-n} g b^{n}) (x) +n.
 \]
This extension (similar to the extension~\eqref{eq:extending} in the proof of Proposition~\ref{p:stab-subgr}) then satisfies $\phi(bgb^{-1})  = \phi(b) \phi(g) \phi(b)^{-1}$, and we complete the extension by taking $\phi(gb^n) = \phi(g) \phi(b)^n$ for every $g\in G, n\in \Z$. This action is by $C^{1+\alpha}$-diffeomorphisms; as the group~$N_4$ has a cyclic center and its generator $c$ acts non-trivially, this action is faithful.
 
 \end{proof}

 \subsection{Topologically free actions}\label{s:top-free}

One might be not fully satisfied with the construction in the proof of Theorem~\ref{thm crit interval} when different actions, provided by Proposition~\ref{p:lower-K}, are combined by a simple concatenation, as it is possible then that some elements act with an interval of fixed points. Such a behaviour is forbidden for the class of topologically free actions:
\begin{definition}
The action of $G$ on $M$ is \emph{topologically free}, if for any $g\in G\setminus \{e\}$ the set of fixed points of the action of~$g$ has no internal points: $\mathrm{Int}(\Fix(g)) = \emptyset$.
\end{definition}

 \begin{example}\label{ex:N3-twice}
The critical regularity of the class of topologically free actions of $N_3\times N_3$ on $[0,1]$ is at most~$1+1/2$. In particular, it is strictly smaller than $\Crit_{N_3\times N_3}([0,1]) = 2$.
 \end{example}
 
 \begin{proof}
Let $a_1,b_1,c_1; a_2,b_2,c_2$, be the generators of the factors $N_3$ of the group $G=N_3\times N_3$, so that 
\[
[a_i,b_i]=c_i, \quad i=1,2.
\]
Then $Z(G)=\langle c_1,c_2\rangle \simeq \Z^2$; also, we have~$[G,G]=Z(G)$.

The following lemma holds.
\begin{lemma}
Let $\phi:G \to \Homeo_+([0,1])$ be a topologically free action of the group $G=N_3\times N_3$. Then there exist a central element $c\in Z(G)$ and its support interval $I_c$, such that the stabilizer subgroup $\Stab(I_c)$ is abelian. 
\end{lemma}
\begin{proof}
As the action is topologically free, in particular, the element $c_1$ acts nontrivially. Take any its support interval~$J$. If $\Stab(J)$ is abelian, we take $c:=c_1$ and $I_c:=J$. Otherwise, there exist two elements $g_1,g_2\in \Stab(J)$ that do not commute. Then, consider $c':=[g_1,g_2]\in Z(G)\cap \Stab(J)$. As the action is topologically free, $\phi(c')$ acts nontrivially on~$J$. On the other hand, as it is a commutator, its action has fixed points inside~$J$, and hence there is a non-trivial support interval $J'\subsetneq J$ of~$c'$.
Again, either $\Stab(J')$ is abelian, or we take two non-commuting $g_1',g_2'\in \Stab(J)$, denote $c'':=[g'_1,g'_2]$ and take a support interval $J''\subsetneq J'$ of~$c''$. 

This procedure stops after a finite number of steps (actually, at most on the third step): one has 
\[
Z(G) \cap \Stab(J) > Z(G) \cap \Stab(J') > Z(G) \cap \Stab(J'') >..., 
\]
with the rank decreasing on each step. Hence, in at most $\rk Z(G)+1=3$ steps we find $c\in Z(G)$ and its support interval $I_c$ with the abelian $\Stab(I_c)$.
\end{proof}

Now, for every abelian group $H<G$ one has $\growth(G/H)\ge 2$. Indeed, without loss of generality we can assume that $Z(G)<H$, as otherwise we can replace $H$ by a larger abelian subgroup $\langle H, Z(G)\rangle$. Now, due to Theorem~\ref{t.B-G}, one can have $\growth(G/H)=1$ only if $d_1^G-d_1^{H;G}=1$, in other words, if the image of $H$ under the map $\varphi_1:G \to G/Z(G)\simeq \Z^4$ is of rank~$3$. Let us show that this is impossible.

Indeed, assume otherwise and consider a function $f:G\times G \to \Z$, defined by 
\[
f(g_1,g_2) = n_1+n_2, \quad \text{where} \,\, [g_1,g_2]=c_1^{n_1} c_2^{n_2}.
\]
This function descends on the quotient by the center for each of the arguments, $f: (G/Z(G))\times (G/Z(G)) \to \Z$, and is a restriction on $G/Z(G)\simeq \Z^4$ of a symplectic form in~$\R^4$. On the other hand, as $H$ is abelian, one has $f(g_1,g_2)=0$ for any $g_1,g_2\in H$. However, a symplectic form in $\R^4$ cannot vanish in restriction on a $3$-dimensional subspace (the span of $H/Z(G)$), and this provides the desired contradiction.

Hence, $\growth(G/\Stab(I_c))\ge 2$, and thus the regularity of the action cannot exceed~$1+1/2$.
 \end{proof}

\begin{small}


\hypersetup{backref=true}

\end{small}

\noindent Victor Kleptsyn. CNRS, Univ Rennes, IRMAR - UMR 6625, F-35000 Rennes, France.

\noindent e-mail: victor.kleptsyn@univ-rennes.fr

\bigskip

\noindent Maximiliano Escayola. School of Mathematics, Korea Institute for Advanced Study (KIAS), Seoul, 02455, Korea.

\noindent e-mail: maxiescayola@kias.re.kr

\end{document}